\documentclass[a4paper,10pt]{amsart}
\usepackage[utf8]{inputenc}
\usepackage{amsthm}
\usepackage{amsmath}
\usepackage{bm}
\usepackage{amsfonts}
\usepackage{amssymb}
\usepackage{mathtools}
\usepackage{mathrsfs}
\usepackage{setspace}
\usepackage{xcolor}
\usepackage{textgreek}
\usepackage{todonotes}
\usepackage{thmtools} 

\usepackage[pdfdisplaydoctitle,colorlinks,breaklinks,urlcolor=blue,linkcolor=blue,citecolor=blue]{hyperref} 

\newtheorem{theorem}{Theorem}[section]

\newtheorem{definition}[theorem]{Definition}

\newtheorem{lemma}[theorem]{Lemma}

\newtheorem{proposition}[theorem]{Proposition}

\newcommand{\uu}{\mathbf{u}}
\newcommand{\vv}{\mathbf{v}}

\newcommand{\N}{\mathbb{N}}
\newcommand{\R}{\mathbb{R}}
\newcommand{\T}{\mathbb{T}}
\newcommand{\Z}{\mathbb{Z}}

\newcommand{\C}{\mathbb{C}}
\newcommand{\PP}{\mathbb{P}}
\newcommand{\dvg}{\mbox{div}\,}
\newcommand{\E}[1]{\mathbb{E}\left[#1\right]}

\theoremstyle{remark}
\newtheorem{remark}[theorem]{Remark}

\numberwithin{equation}{section}

\title[Quantitative Mixing and dissipation enhancement OU]{Quantitative mixing and dissipation enhancement property of Ornstein-Uhlenbeck flow}
\author[U. Pappalettera]{Umberto Pappalettera}
  \address{Scuola Normale Superiore, Piazza dei Cavalieri, 7, 56126 Pisa, Italia}
  \email{\href{mailto:umberto.pappalettera@sns.it}{umberto.pappalettera@sns.it}}
\date\today

\begin{document}

\begin{abstract}
This work deals with mixing and dissipation enhancement for the solution of advection-diffusion equation driven by a Ornstein-Uhlenbeck velocity field. 
We are able to prove a quantitative mixing result, uniform in the diffusion parameter, and enhancement of dissipation over a finite time horizon. 
\end{abstract}

\maketitle

\section{Introduction}

Let $\T^d = \R^d / 2\pi \Z^d$ be the $d$-dimensional torus, $d\geq 3$.
In this paper we study dissipation enhancement and mixing in the advection-diffusion of a passive scalar $T(t,x)$, given by the following partial differential equation:
\begin{align} \label{eq:temp}
\partial_{t}T + \uu \cdot \nabla T  = \kappa \Delta T
\quad \mbox{ in } [0,1] \times \T^d,
\end{align}
with initial value $T |_{t=0} = T_0 \in L^2(\T^d)$.
In the equation above, $\uu(t,x)$ is a random divergence-free velocity field which models in a very simplified fashion an incompressible turbulent fluid, and the parameter $\kappa \in [0,1]$ is called \emph{molecular diffusivity}.

Roughly speaking, our aim is to prove that the effective diffusion in \eqref{eq:temp} is enhanced due to the random turbulent transport $\uu$, and for a suitable choice of the velocity field $\uu$ we are able to provide a quantitative mixing estimate, that is a bound in a space of distributions on $\T^d$ for the difference between the (random) solution $T$ of \eqref{eq:temp} and the (deterministic) solution $\bar{T}$ of
\begin{align} \label{eq:temp_bar}
\partial_t \bar{T} = (\kappa \Delta + \mathcal{L}) \bar{T}
\quad \mbox{ in } [0,1] \times \T^d,
\end{align}
with initial value $ \bar T |_{t=0} = T_0$, where $\mathcal{L}$ is the second-order negative-semidefinite operator, usually called \emph{eddy diffusivity}, rigorously defined by \eqref{eq:def_L} below.

A similar problem has recently been investigated in \cite{FlGaLu21}, where a Gaussian,  white-in-time velocity field $\uu$ in a bounded domain $D \subset \R^d$ has been considered, interpreting the advection term $\uu \cdot \nabla T$ as a Stratonovich stochastic integral and imposing Dirichlet boundary conditions on the solution. In that work, the authors are able to prove the following estimate:
\begin{align*}
\E{
\sup_{t \in [0,1]}
\left(
\int_D T(t,x) \phi(x) - \int_D \bar{T}(t,x) \phi(x)
\right)^2}
\leq
\frac{\epsilon_Q}{2\kappa} \|T_0\|_{L^2(D)}^2 \|\phi\|_{L^\infty(D)}^2,
\end{align*}
where $\phi \in L^\infty(D)$ is a test function and the quantity $\epsilon_Q$ is a small parameter related to the spatial covariance of $\uu$.
Their result is based on a scaling limit procedure first introduced in \cite{Ga20}, where it is shown how a suitable sequence of Stratonovich transport noises added to a hyperbolic transport equation can produce a parabolic equation in the limit, see also \cite{FlLu20,FlGaLu21b,FlGaLu21c}.

However, the noise considered in \cite{FlGaLu21} is arguably not physical. 
As already discussed in \cite[Section 4]{MaKr99}, the assumption of the velocity field $\uu$ being white-in-time is far from realistic in actual turbulent fluids. 
One main problem is that such a $\uu$ is not even a function -- it only takes values in a space of distributions.

The main motivation behind this work is to overcome this issue by replacing the white-in-time velocity field considered in \cite{FlGaLu21} with a velocity field $\uu^\alpha$, $\alpha>1$, with finite correlation time of the following form
\begin{align*}
\uu^\alpha(t,x) \coloneqq \sum_{j \in J} \uu_j(x) \eta^{\alpha,j}(t),
\end{align*}
where $J$ is an index set, $\{\uu_j\}_{j \in J}$ is a family of suitable time-independent vector fields, and $\{\eta^{\alpha,j}\}_{j \in J}$ is a family of i.i.d. stationary Ornstein-Uhlenbeck processes with covariance $\mbox{Cov}\left(\eta^{\alpha,j}(t),\eta^{\alpha,j}(s) \right)=\frac{\alpha}{2}  \exp(-\alpha|t-s|)$.
In what follows, the parameter $\alpha$ will be taken large.
Roughly speaking, it corresponds to the inverse of the typical turnover period of the fluid, as can be seen from the expression of the covariance between $\eta^{\alpha,j}(t)$ and $\eta^{\alpha,j}(s)$.
In particular, the larger the $\alpha$, the more rapid and intense the fluctuations of the fluid are, and the better the mixing is (cfr. the statement of \autoref{thm:main} below). 
Also, the time covariance $\frac{\alpha}{2}  \exp(-\alpha|t-s|)$ converges as distributions to the Dirac delta as $\alpha \to \infty$, and indeed we shall recover the same result of \cite{FlGaLu21} in the limit. The key point here is that iterated integral of the Ornstein-Uhlenbeck process, when tested against smooth functions, behaves as the iterated stochastic integral (in the sense of Stratonovich) of a white noise, plus an error of order $\alpha^{-1}$.
Compared to \cite{FlGaLu21}, where the operator $\mathcal{L}$ is interpreted as a Stratonovich-to-It\=o corrector and thus comes out naturally from computations, one of the main technical difficulties of this paper is to understand \emph{quantitatively} this asymptotic behaviour.

The choice of our model, admittedly phenomenological, stems in the literature concerning stochastic model reduction of geophysical systems with two different time scales (see \cite{MaTiVa01} and references therein), and has proved itself to be very convenient because of the full computability of the Ornstein-Uhlenbeck process.
Ideally, one would take $\uu$ as the solution of the Navier-Stokes equations, or a slight modification thereof (see for instance \cite{FlPa21+,DePa22+} and the series of papers \cite{BeBlPu18+,BeBlPu19+,BeBlPu20,BeBlPu20+}), but this seems currently out of the scope of our technique.  

In the following, we denote $H^s(\T^d)$ the homogeneous Sobolev space on $\T^d$, with smoothness parameter $s\in \R$, referring to \autoref{ssec:anal_fun} for details.
In order to study the mixing properties of the velocity field $\uu^\alpha$, we introduce the following parameters related to the spatial structure of the velocity field. 
Let
\begin{align} \label{eq:def_eps}
\epsilon^2
&\coloneqq
\sup \left\{
\sum_{j \in J}\langle \uu_j , \vv \rangle^2 
\, : \,
\vv \in L^2(\T^d,\R^d), \|\vv\|_{L^2(\T^d,\R^d)}=1
\right\},
\end{align}
and, for every $\gamma \in (0,(d-2)/6)$, $d \geq 3$, define
\begin{align} \label{eq:def_mu}
\mu \coloneqq
\max \left\{
|J|
\sum_{j \in J}  \| \uu_j \|_{H^{d/2-\gamma}(\T^d)} \, ,\,
|J|
\sum_{j \in J}  \| \uu_j \|_{L^\infty(\T^d)}
\, ,\,
\| \mathcal{L} \|^{1/2}_{H^2(\T^d),L^2(\T^d)}
\right\}.
\end{align}
where $\mathcal{L}$ is defined as the operator on the space of zero-mean distributions $\mathcal{S}'(\T^d)$ acting on smooth functions $f \in C^\infty(\T^d)$ via the formula
\begin{align} \label{eq:def_L}
(\mathcal{L} f)(x)
&\coloneqq
\frac12 \sum_{j \in J}
\uu_j(x) \cdot \nabla (\uu_j \cdot \nabla f)(x).
\end{align}
Notice that both $\epsilon$ and $\mu$ only depend on the space structure of the noise. 
We will assume without loss of generality $\mu \in (1,\infty)$.

Our main result is the following, and allows to control the distance between $T$ and $\bar{T}$ in a space of distribution-valued $\theta$-H\"older functions: 
\begin{theorem}[Quantitative mixing] \label{thm:main}
Let $T_0 \in L^2(\T^d)$ with zero mean, $d\geq 3$, and denote $T$ the solution of \eqref{eq:temp} with $\uu=\uu^\alpha$. Then, for every $\gamma \in (0,(d-2)/6)$ and $s>0$ there exist coefficients $\theta, \varkappa$, $\varsigma > 0$ such that for every $\alpha$ sufficiently large
\begin{align*}
\E{
\left\| T - \bar{T} \right\|_{C^\theta([0,1],H^{-s}(\T^d))}
}
\leq
C
\|T_0\|_{L^2(\T^d)}
\left(
\epsilon
+
\frac{\mu^{2+\gamma}}{\alpha^\varkappa}
\right)^\varsigma,
\end{align*}
where $\bar{T}$ is the solution of \eqref{eq:temp_bar}, $\epsilon$, $\mu$ are defined by \eqref{eq:def_eps} and \eqref{eq:def_mu} and $C \in (0,\infty)$ is an unimportant constant, not depending on the parameters $T_0,\kappa,s,\theta,\varkappa,\varsigma,\alpha,\epsilon,\mu$.
\end{theorem}

In the statement of the theorem, by {\em $\alpha$ sufficiently large} we mean more precisely: there exist  positive numbers $p_1,p_2>0$, depending only on the parameters $\gamma, \theta$ and $\varkappa$, such that the thesis holds for every $\alpha$ satisfying $\mu \alpha^{-p_1} \log^{p_2}(1+\alpha)<1$. The coefficients $p_1,p_2$ can be computed explicitly, from a close inspection of the proof of \autoref{prop:discr_holder} in \autoref{sec:proof}. 

Notice that the condition $d \geq 3$ is only required by technical motivations. The case $d=2$ can be taken into account considering systems in the three dimensional torus with translation invariance along one direction, see also \autoref{rmk:d=2}. 

The result stated in previous \autoref{thm:main} is far from obvious. It is well-known \cite{MaKr99} that dissipation enhancement for advection-diffusion equations occurs for the average field $\E{T}$, when the fluid is sufficiently turbulent, and therefore $\E{T}$ is indeed expected to solve an equation of the form \eqref{eq:temp_bar} -- which can be rigorously proved when $\uu$ is delta-correlated in time, interpreting $\mathcal{L}$ as a Stratonovich-to-It\=o corrector. On the other hand, here not only we deal with a more realistic non-delta-correlated in time velocity field $\uu$, but we also show that the actual field $T$, and not only its average, is close to the solution of \eqref{eq:temp_bar}. So, our result can be seen as an enhancement of both dissipation and mixing. 

In addition, we remark that our estimate holds \emph{uniformly} in the molecular diffusivity $\kappa \in [0,1]$, that is in sharp contrast with the result of \cite{FlGaLu21}. This feature is not secondary, as it is known that diffusion can limit the mixing rate in certain examples \cite{MiDo18}, and we refer to \cite{BeBlPu20} for a thorough discussion on the topic. Notice that our mixing result is similar to \cite[Theorem 1.1]{FlGaLu21+} for two dimensional Euler and Navier-Stokes equations with transport noise.

The mixing property stated in \autoref{thm:main} has many consequences, among which transfer on energy to small scales stands out.
This mechanism is the main responsible for the enhanced dissipation of $L^2(\T^d)$ norm for the solution $T$ of \eqref{eq:temp} -- see \autoref{rmk:transfer} below and \cite{FeIy19,CZDeEl20}. 

In what follows, denote for simplicity $A$ the operator on the space of zero-mean distributions $\mathcal{S}'(\T^d)$ acting on smooth functions $f \in C^\infty(\T^d)$ as $Af = \kappa \Delta f+ \mathcal{L}f$.
Our second result permits to estimate the rate of decay of the $L^2(\T^d)$ norm of the solution of \eqref{eq:temp}, when the molecular diffusivity $\kappa$ is strictly positive.
 
\begin{theorem}[Enhanced dissipation of $L^2(\T^d)$ norm] \label{thm:enhanced}
In the same setting as \autoref{thm:main}, assume in addition $\kappa>0$. Let $c = C^{1/2}\left(\epsilon+\frac{\mu^{2+\gamma}}{\alpha^\varkappa}\right)^{\varsigma/2}>0$ and denote $\lambda>0$ the principal eigenvalue of the operator $-A$.
Then the following inequality holds with probability at least $1-c$ for every $t \in [0,1]$
\begin{align*}
\| T_t \|_{L^2(\T^d)}
\leq 
\frac{ \| T_0 \|_{L^2(\T^d)}}
{
\left( 1+\frac{\kappa}{2 \lambda c^2}\log \left( \frac{c^2 e^{2 \lambda t} + 1}{c^2+1} \right)  \right)^{1/2}
}.
\end{align*}
In particular, for every $t \in [0,1]$ it holds
\begin{align} \label{eq:dissipation}
\E{\| T_t \|_{L^2(\T^d)}}
\leq
c\| T_0 \|_{L^2(\T^d)}
+
\frac{\| T_0 \|_{L^2(\T^d)}}
{
\left( 1+\frac{\kappa}{2 \lambda c^2}\log \left( \frac{c^2 e^{2 \lambda t} + 1}{c^2+1} \right)  \right)^{1/2}
}.
\end{align}
\end{theorem}

A few remarks are in order.
It is well-known that the only estimate available a priori for the $L^2(\T^d)$ norm of $T$ is given by (see \autoref{ssec:solution} below):
\begin{align*}
\|T_t\|_{L^2(\T^d)} \leq e^{-\kappa t} \|T_0\|_{L^2(\T^d)},
\end{align*}
and the previous inequality is in fact an equality in the inviscid case $\kappa=0$, when dissipation does not occur.
The content of our previous theorem can thus be read as follows: for every fixed $t>0$, if $\kappa>0$ and $\uu^\alpha$ is such that $\lambda\gg1$ and $c \ll 1$ then $\E{\|T_t\|_{L^2(\T^d)}}\ll e^{-\kappa t} \|T_0\|_{L^2(\T^d)}$, namely dissipation of the $L^2(\T^d)$ norm is enhanced.
Also, taking formally $c \to 0$ in \eqref{eq:dissipation} we obtain an augmented decay rate for the $L^2(\T^d)$ norm, that is
\begin{align} \label{eq:enhanced_remark}
\E{\| T_t \|_{L^2(\T^d)}} \leq
\frac{\lambda^{1/2}}{\kappa^{1/2}} e^{- \lambda t}\| T_0 \|_{L^2(\T^d)},
\qquad
\forall t \in [0,1].
\end{align}

In order to make $\lambda \gg 1$ and $c \ll 1$ \emph{simultaneously}, one can choose first the family $\{ \uu_j \}_{j \in J}$ so that $\lambda \gg 1$ and $\epsilon \ll 1$ at the same time, and then take $\alpha$ sufficiently large so that $c \ll 1$. Indeed, the parameters $\lambda$, $\epsilon$ and $\mu$ depend only on spatial properties of the family $\{ \uu_j \}_{j \in J}$, and they are essentially decoupled from $\alpha$, which models instead the temporal structure of the velocity field.
The problem of finding a family $\{ \uu_j \}_{j \in J}$ that renders simultaneously $\lambda$ large and $\epsilon$ small has been previously treated (see \cite{Ga20} and subsequent works), and in the present work we will not discuss this point in details.
Let us just mention that in \cite{FlGaLu21} the authors use a suitable structure of the noise inspired by the Kraichnan model of turbulence, cfr. also \cite[Remark 1.8]{FlGaLu21+} and the recent work \cite{GeYa22+} for a similar construction.
Also worth of mention is that, more generally, \cite{FlGaLu21+} allows for a space structure of the noise as in the series of papers starting with \cite{Ga20} on the convergence of (stochastic) transport equations to (deterministic) parabolic ones.

Our main results are only significant over a finite time horizon, that we have taken of unitary length for simplicity: indeed, in \autoref{thm:main} the relative error between $T_t$ and $\bar{T}_t$ becomes more and more important as time increases, since $\bar{T}_t \to 0$ in $L^2(\T^d)$ exponentially fast; similarly, the estimate of \autoref{thm:enhanced} is extremely good for small times, but as it is stated becomes weaker than the simpler $\|T_t\|_{L^2(\T^d)} \leq e^{-\kappa t} \|T_0\|_{L^2(\T^d)}$ for long times. 

The differences between our result and the existing literature are several.
Some results on dissipation enhancement properties of random velocity fields have already been obtained in the framework of homogenization theory. 
In \cite{CaFo95,Fa04}, it is proved that, for a suitable sequence of random velocity fields $\uu$, the solution $T$ of \eqref{eq:temp} converges in law towards a new advection-diffusion equation, with augmented diffusivity but non-zero advection. An equation of the form \eqref{eq:temp_bar} is then recovered when considering the average $\E{T}$ or an appropriate diffusive scaling of the solution.
However, these works do not show any mixing property. 

Also interesting is the outstanding  \cite{CoKiRyZl06}, where a deterministic time-independent velocity field of the form $a \uu$, $a\gg 1$ is considered. Roughly speaking, in that work the authors show that the system is well-mixed for $a$ sufficiently large if and only if the operator $\uu \cdot \nabla$ has no eigenfunctions in $H^1(\T^d)$ other than the constant functions, but a quantitative estimate of the mixing property, as that given in \autoref{thm:main} or \cite[Theorem 1.1]{FlGaLu21}, is not provided. In \cite{BeHaNa05}, a similar result is obtained (for bounded domains with Dirichlet boundary conditions) studying the principal eigenvalue $\lambda_a$, $a \gg 1$, of the operator $-\kappa \Delta + a\uu \cdot \nabla$.

Another example of results similar to ours in the  existing literature is given by \cite{BeBlPu20}, where the authors consider advection-diffusion equation on $\T^d$, $d=2,3$, subject to a velocity field $\uu$ solving 2D stochastic Navier-Stokes or 3D stochastic hyperviscous Navier-Stokes.
Roughly speaking, the main result of \cite{BeBlPu20} states that for every $\kappa \in (0,1]$ and $s>0$ there exists $\gamma>0$ and a random variable $C$ with values in $(0,\infty)$ such that, for every zero mean initial datum $T_0 \in H^s(\T^d)$ the solution $T$ of \eqref{eq:temp} satisfies \emph{almost surely} for every $t>0$
\begin{align}
\|T_t\|_{H^{-s}(\T^d)} \label{eq:BBP1}
&\leq
C e^{-\gamma t} 
\|T_0\|_{H^{s}(\T^d)},
\\
\|T_t\|_{L^2(\T^d)}  \label{eq:BBP2}
&\leq
C \kappa^{-1}
e^{-\gamma t} 
\|T_0\|_{L^2(\T^d)},
\end{align}
with some bound on the moments of $C$ independent of $\kappa$; moreover, \eqref{eq:BBP1} holds even in the inviscid case $\kappa=0$. Comparing \eqref{eq:BBP1} with \autoref{thm:main}, notice that \eqref{eq:BBP1} implies only the exponential decay of $H^{-s}(\T^d)$ norm of $T$, while \autoref{thm:main} gives the additional information that $T$ decays staying close to $\bar{T}$ (decaying of $\bar{T}$ is quantified in \autoref{ssec:solution}).
Moreover, both \eqref{eq:BBP1} and \eqref{eq:BBP2} become effective for large times; on the contrary, our \autoref{thm:main} and \autoref{thm:enhanced} permit to augment both dissipation and mixing at \emph{fixed} time $t>0$, provided that the velocity field $\uu^\alpha$ is chosen in a suitable way.

\subsection{Structure of the paper}
The paper is structured as follows.

In \autoref{sec:def} we recall basic definitions about homogeneous Sobolev spaces and some classical  result concerning well-posedness of \eqref{eq:temp} and \eqref{eq:temp_bar}. 
A precise description of our model is carried on in \autoref{ssec:model}. 

In \autoref{sec:est} we collect some auxiliary results concerning the time increments of the process $T$ solution of \eqref{eq:temp}. There we prove that $T$ has a.s. regular trajectories as a process taking value in a space of distributions, and we obtain good bounds for the expectation of its increments – see \autoref{lem:increments_T} and \autoref{lem:increments_T_bis}. 
In addition, we state the key result \autoref{prop:discr_holder}, needed for the proof of our main theorem, that consists in a bound on the quantity
\begin{align*}
\sup_{\substack{n,m=1,\dots,1/\delta-1\\n>m}}
(n\delta-m\delta)^{-\theta}
\left| 
\langle \phi, T_{n\delta} \rangle
-
\langle \phi, T_{m\delta} \rangle
-
\delta \sum_{k=m}^{n-1}
\langle A \phi, T_{k\delta} \rangle 
\right|,
\end{align*}
where $\delta$ is a small parameter, suitably chosen in depending on $\alpha$.

In \autoref{sec:main}, we give the proof of \autoref{thm:main} and \autoref{thm:enhanced}. 
  
Concerning the first result, the key idea consists in introducing the random distribution $f$ defined via the formula
\begin{align*}
\langle \phi, f_t \rangle
=
\langle \phi, T_t \rangle
-
\langle \phi, T_0 \rangle
-
\int_0^t \langle A\phi, T_s \rangle ds,
\quad \forall \phi \in H^2(\T^d),
\quad \forall t \in [0,1]:
\end{align*}
of course if one could prove $f=0$, then one would have $T_t = \bar T_t$ for every $t \in [0,1]$; however, it turns out that the difference $T - \bar{T}$ depends \emph{path-by-path} continuously on $f$,  (\autoref{lem:mild_solution}) and therefore we are able to deduce an estimate on $T - \bar{T}$ from an estimate on $f$ --  cfr. \autoref{prop:holder_f}.

As for \autoref{thm:enhanced}, its proof relies on the following energy inequality
\begin{align*}
\frac{d}{dt} \|T_t\|_{L^2(\T^d)}^2
\leq
-2\kappa \frac{\|T_t\|_{L^2(\T^d)}^4}{\|T_t\|^2_{H^{-1}(\T^d)}}
\end{align*}
and a bound on $\|T_t\|^2_{H^{-1}(\T^d)}$ obtained applying \autoref{thm:main} with $s=1$. Since our result on mixing is truly quantitative, we are able to say that $\|T_t\|^2_{H^{-1}(\T^d)}$ is smaller than a certain threshold (depending on $t$) with high probability, and using this information in the previous inequality we deduce an explicit rate of decay for the $L^2(\T^d)$ norm of $T$.

Finally, in \autoref{sec:proof} we prove  \autoref{prop:discr_holder}. Its proof is based on a discretization procedure very common in the literature about averaging and Wong-Zakai approximations theorems for stochastic differential equations, and are inspired by our previous works \cite{AsFlPa20+,FlPa21}.

All in all, in this work there coexist both analytic and probabilistic techniques: the advection-diffusion equation \eqref{eq:temp} is interpreted in a pathwise sense, although the velocity field $\uu$ is random -- it is usually referred to as random PDE rather than SPDE; the key \autoref{lem:mild_solution} is analytic as well; the results contained in \autoref{sec:proof} are more probabilistic in the spirit, and rely on explicit computability of the Ornstein-Uhlenbeck process, Doob maximal inequality, Burkholder-Davis-Gundy inequality and Kolmogorov continuity criterion.

\section{Notation and preliminaries} \label{sec:def} 

Hereafter, we use the symbol $\lesssim$ to indicate inequality up to an unimportant constant $C \in (0,\infty)$.
If the constant $C$ depends on parameters $p_1,\dots,p_n$ we use the symbol $\lesssim_{p_1,\dots,p_n}$ instead. 

\subsection{Functional analytic preliminaries} \label{ssec:anal_fun}

In this paragraph we give basic definitions and first results concerning Sobolev spaces on the torus. 
Here we only consider Sobolev spaces of scalar functions and distributions, the vector-valued case being easily handled componentwise.    

First we introduce some notations.
Let $e_k(x) = (2\pi)^{-d/2} e^{i k \cdot x}$, $k \in \Z^d$. The set $\{e_k\}_{k \in \Z^d}$ is a complete orthonormal system of $L^2(\T^d,\C)$ made of eigenfunctions of the Laplace operator: $\Delta e_k = -|k|^2 e_k$ for every $k \in \Z^d$.

A generic function $f \in L^2(\T^d,\C)$ can be then represented as a Fourier series:
\begin{align*}
f(x) = \sum_{k \in \Z^d} \hat{f}_k e_k(x),
\quad
x \in \T^d,
\end{align*}
for a (unique) square summable sequence $\{\hat{f}_k \}_{k \in \Z^d}$ of Fourier coefficients. The map $\mathfrak{F}:L^2(\T^d,\C) \to \ell^2(\Z^d,\C)$, that associates to every $f$ the sequence of its Fourier coefficients, is an isomorphism of Hilbert spaces; it can be extended to an isomorphism between the space of (complex-valued) tempered distributions $\mathcal{S}'(\T^d,\C)$ and space of sequences of Fourier coefficients $\{ \hat{f}_k \}_{k \in \Z^d}$ such that $\sum_{k \in \Z^d} |k|^{2s} |\hat{f}_k|^2 < \infty$ for some $s \in \R$. 
Without explicit mention, in the following we will identify a tempered distribution $f \in \mathcal{S}'(\T^d,\C)$ with the sequence of its Fourier coefficients $\{ \hat{f}_k \}_{k \in \Z^d}$. 

Unless otherwise stated, we will only work in the following with the space $\mathcal{S}'(\T^d) \subset\mathcal{S}'(\T^d,\C)$ of real-valued distributions with zero mean, that corresponds to sequences $\{ \hat{f}_k \}_{k \in \Z^d}$ such that $\hat{f}_0=0$ and $\hat{f}_k = \overline{\hat{f}_{-k}}$ for every $k \in \Z^d$. 
We denote $\Z^d_0 = \Z^d \setminus \{ \mathbf{0}\}$.

Smoothness of a distribution $f \in \mathcal{S}'(\T^d)$ can be read in terms of summability of the sequence of its Fourier coefficients $\{ \hat{f}_k \}_{k \in \Z^d_0}$: for $s \in \R$, define the Sobolev space
\begin{align*}
H^s(\T^d) = 
\left\{ 
f \in \mathcal{S}'(\T^d) :
\| f \|_{H^s} \coloneqq
\sum_{k \in \Z^d_0} |k|^{2s} |\hat{f}_k|^2 < \infty
\right\},
\end{align*}
which is a Hilbert space when equipped with the scalar product 
\begin{align*}
\langle f, g \rangle_{H^s(\T^d)}
=
\sum_{k \in \Z^d_0} |k|^{2s} \hat{f}_k \hat{g}_{-k},
\quad
f,g \in H^s(\T^d)
.
\end{align*}
In the special case $s=0$, the Sobolev space corresponds to the space $L^2(\T^d)$ of real-valued, square integrable functions on the torus with zero mean. The scalar product $\langle \cdot, \cdot \rangle \coloneqq \langle \cdot, \cdot \rangle_{H^0(\T^d)}$ is also a duality map between $H^s(\T^d)$ and $H^{-s}(\T^d)$ for every $s \in \R$:
\begin{align*}
\langle f, g \rangle
=
\sum_{k \in \Z^d_0} \hat{f}_k \hat{g}_{-k},
\quad
f \in H^s(\T^d), \,
g \in H^{-s}(\T^d)
.
\end{align*}
  
Sobolev spaces form a Hilbert scale in the sense of Krein-Petunin \cite{KrPe66}, with respect to the operator $(-\Delta)^{1/2}$: 
\begin{align*}
\langle f, g\rangle_{H^s} = 
\langle (-\Delta)^{s/2} f,(-\Delta)^{s/2} g \rangle. 
\end{align*}
In particular, the map $(-\Delta)^{s/2}: H^r(\T^d) \to H^{r+s}(\T^d)$ is an isomorphism for every $s,r \in \R$ and the following interpolation inequality holds between $H^{s_1}(\T^d)$ and $H^{s_2}(\T^d)$, for $s_1, s_2 \in \R$, $s_1 < s_2$ and $\theta \in (0,1)$:
\begin{align} \label{eq:interp}
\|f\|_{H^{s_\theta}(\T^d)}
\leq
\|f\|_{H^{s_1}(\T^d)}^\theta
\|f\|_{H^{s_2}(\T^d)}^{1-\theta},
\quad
s_\theta = \theta s_1 + (1-\theta)s_2.
\end{align}

Finally, we have the following lemma, cfr. \cite[Corollary 2.55]{BaChDa11} for a proof in the full space:
\begin{lemma} \label{lem:prod_bess}
Let $s_1,s_2 \in \R$ such that $s_1,s_2 < d/2$ and $s_1+s_2 > 0$. Then for every $f \in H^{s_1}(\T^d)$ and $g \in H^{s_2}(\T^d)$ the product $fg \in H^{s_1+s_2-d/2}(\T^d)$ and
\begin{align*}
\|fg\|_{H^{s_1+s_2-d/2}(\T^d)}
\lesssim_{s_1,s_2}
\|f\|_{H^{s_1}(\T^d)} \|g\|_{H^{s_2}(\T^d)}.
\end{align*}
\end{lemma}

\begin{remark} \label{rmk:d=2}
Condition $d\geq 3$ stated in the introduction is a technical limitation of our method, due to application of \autoref{lem:prod_bess}, but there is no physical reason for this constraint. 
Also, the case $d=2$ in \autoref{thm:main} and \autoref{thm:enhanced} is readily implied by our results in dimension $d=3$ and the following observation: for every $s \in \R$ and $f \in H^s(\T^2)$, the function $g:\T^3 \to \R$ defined by $g(x_1,x_2,x_3) \coloneqq f(x_1,x_2)$ satisfies for every $k=(k_1,k_2,k_3) \in \Z^3_0$
\begin{align*}
\hat{g}_{(k_1,k_2,k_3)}
=
\begin{cases}
\hat{f}_{(k_1,k_2)} &\mbox{ if } k_3 = 0,
\\
0 &\mbox{ if } k_3 \neq 0,
\end{cases}
\end{align*}
and thus $g \in H^s(\T^3)$ with $\|g\|_{H^s(\T^3)}=\|f\|_{H^s(\T^2)}$.
\end{remark}

\subsection{The model} \label{ssec:model}
\subsubsection*{Stationary Ornstein-Uhlenbeck processes}
Let $J$ be a finite index set of cardinality $|J|$, and let $(\Omega^+,\{\mathcal{F}^+_t\}_{t \geq 0},\PP^+)$ 
and $(\Omega^-,\{\mathcal{F}^-_t\}_{t \geq 0},\PP^-)$ be two filtered probability spaces, satisfying the usual conditions, which support two families of i.i.d. Brownian motions $\{W^{+,j}\}_{j \in J}$ and $\{W^{-,j}\}_{j \in J}$.
Set $W^j_t = W^{+,j}_t$, for $t \geq 0$, and $W^j_t = W^{-,j}_{-t}$, for $t<0$. 

For every $\alpha > 1$, the processes
\begin{equation} \label{eq:defin_etaa}
\eta^{\alpha,j}(t) 
\coloneqq 
\int_{-\infty}^t \alpha e^{-\alpha(t-s)} dW^j_s, 
\quad t \in [0,1], \quad j \in J,
\end{equation}
constitute a family of i.i.d. stationary Ornstein-Uhlenbeck processes solutions of
\begin{align*}
d \eta^{\alpha,j} = - \alpha \eta^{\alpha,j} dt + \alpha dW^j_t, 
\quad t \in [0,1], \quad j \in J,
\end{align*}
on the filtered probability space $(\Omega,\{\mathcal{F}_t\}_{t \in [0,1]},\PP)$ defined by $\Omega \coloneqq \Omega^- \times \Omega^+$, $\PP \coloneqq \PP^- \otimes \PP^+$, and where $\{\mathcal{F}_t\}_{t \geq 0}$ is the augmentation of the filtration
$\{\mathcal{F}^-_\infty \otimes \mathcal{F}^+_t\}_{t \in [0,1]}$.
Notice that this filtration satisfies the usual conditions.
Moreover, by \eqref{eq:defin_etaa} it holds for every $t \in [0,1]$,
\begin{align} \label{eq:explicit_OU}
\eta^{\alpha,j}(t) 
&= 
- \int_0^\infty \alpha e^{-\alpha(t+s)} dW^{-,j}_s
+ \int_0^t \alpha e^{-\alpha(t-s)} dW^{+,j}_s
\\
&= \nonumber
e^{-\alpha t} \eta^{\alpha,j}(0)
+
\alpha \int_0^t  e^{-\alpha (t-s)} dW^j_s,
\end{align}
where in the second line we have used $W^{+,j}_s=W^{j}_s$ for every $s \geq 0$ and
\begin{align*}
\eta^{\alpha,j}(0) = -\int_0^\infty \alpha e^{-\alpha s} dW^{-,j}_s.
\end{align*}
In particular, the previous line shows that $\eta^{\alpha,j}(0)$ is a $\mathcal{F}_0$-measurable random variable independent of $\eta^{\alpha,j'}$ for $j' \neq j$, and it is a centered Gaussian with covariance $\alpha/2$ for every $j \in J$.

\subsubsection*{The velocity field $\uu^\alpha$} 
Next, we consider \eqref{eq:temp} with velocity field $\uu=\uu^\alpha$, $\alpha>1$, given by
\begin{align*}
\uu^\alpha(t,x) = \sum_{j \in J} \uu_j(x) \eta^{\alpha,j}(t),
\end{align*}
where $\{\uu_j\}_{j \in J}$ is a family of smooth vector fields $\uu_j: \T^d \to \R^d$ such that $\dvg \uu_j = 0$ for every $j \in J$. 

The eddy diffusivity operator $\mathcal{L}$ is related to the spatial properties of the family $\{\uu_j\}_{j \in J}$, and it is rigorously defined as the operator on $\mathcal{S}'(\T^d)$ acting on a smooth function $f \in C^\infty(\T^d)$ via the formula \eqref{eq:def_L}.
In general, $\mathcal{L}$ is a negative-semidefinite operator; under the additional assumption
\begin{align} \label{eq:ass_a}
a(x) \coloneqq \sum_{j \in J} \uu_j(x) \uu^t_j(x)=2 \bar \kappa I, 
\quad
\bar \kappa\geq 0,
\end{align}
the eddy diffusivity is a multiple of the Laplace operator: $\mathcal{L}= \bar{\kappa} \Delta$. However, we will not make use of \eqref{eq:ass_a} in the following.

\begin{remark}
Compare our definition of $\epsilon$ \eqref{eq:def_eps} with that of $\epsilon_Q$ in \cite{FlGaLu21}, cfr. (1.2) therein: since the covariance between $\uu^\alpha(t,x)$ and $\uu^\alpha(t,y)$ equals $\frac{\alpha}{2} \sum_{j \in J} \uu_j(x)\otimes\uu_j(y)$ for every $t \in [0,1]$, we have for every $\vv \in L^2(\T^d,\R^d)$ with $\|\vv\|_{L^2(\T^d,\R^d)}=1$
\begin{align*}
\sum_{j \in J}\langle \uu_j , \vv \rangle^2
&=
\sum_{j \in J} \int_{\mathbb{T}^d} \uu_j(x) \cdot \vv(x) dx \int_{\mathbb{T}^d} \uu_j(y) \cdot \vv(y) dy
\\
&=
\sum_{j \in J} \int_{\mathbb{T}^d}\int_{\mathbb{T}^d}
\vv(x)^\top (\uu_j(x)\otimes\uu_j(y)) \vv(y) dx dy.
\end{align*}
Hence, denoting $Q(x,y)=\uu_j(x)\otimes\uu_j(y)$ the two definitions coincide (up to the square and a multiplicative factor depending on $\alpha$). 
\end{remark} 

\subsection{Notion of solution to \eqref{eq:temp} and \eqref{eq:temp_bar}} \label{ssec:solution}

We give now the notion of solution for the random PDE \eqref{eq:temp}
\begin{align*}
\partial_{t}T + \uu^\alpha \cdot \nabla T  = \kappa \Delta T
\quad \mbox{ in } [0,1] \times \T^d,
\end{align*}
with (deterministic) initial value $T |_{t=0} = T_0 \in L^2(\T^d)$. We shall use this notion of solution throughout the paper.
In what follows, $\mathcal{D}(\T^d)$ stands for the space of real-valued, zero-mean, smooth test functions.

\begin{definition} \label{def:sol}
Assume $\kappa>0$. A stochastic process $T:\Omega \times [0,1] \to L^2(\T^d)$, adapted to the filtration $\{\mathcal{F}_t\}_{t \in [0,1]}$, is a (analytically weak, probabilistically strong) solution of \eqref{eq:temp} if there exists a full-measure set $\tilde{\Omega} \subset \Omega$ such that for every $\omega \in \tilde{\Omega}$ it holds: $T(\omega, \cdot) \in L^\infty([0,1], L^2(\T^d)) \cap L^2([0,1],H^1(\T^d))$ and
\begin{align*}
\langle \phi , T_t \rangle
&=
\langle \phi , T_s \rangle
+
\int_s^t 
\langle \uu^\alpha(r) \cdot \nabla \phi, T_r \rangle dr
+
\kappa
\int_s^t
\langle \Delta \phi, T_r \rangle dr,
\end{align*} 
for every $s<t$ and $\phi \in \mathcal{D}(\T^d)$. 
\end{definition}

For every initial datum $T_0 \in L^2(\T^d)$ and $\kappa>0$, \eqref{eq:temp} is well-posed in the sense of the previous definition. Indeed, the advection velocity $\uu^\alpha$ is almost surely smooth in space and H\"older continuous in time, thus by \cite[Corollary 2.2]{Pa83} the only solution is given by Duhamel formula
\begin{align*}
T_t = e^{\kappa \Delta t} T_0 + \int_0^t e^{\kappa \Delta (t-s)} \,( \uu^\alpha(s) \cdot \nabla T_s ) \,ds.  
\end{align*} 
Existence of (probabilistically strong) solution can be proved by standard approximation schemes and Yamada-Watanabe theorem.
Also, by linearity of the equation and mollification we can also check that the map
$t \mapsto \|T_t\|^2_{L^2(\T^d)}$ has a.s. absolutely continuous trajectories and the following energy inequality (as variables in $L^1([0,1])$) holds with probability one:
\begin{align} \label{eq:energy_equality}
\frac{d}{dt} \| T_t \|^2_{L^2(\T^d)}
=
-2\kappa  \| T_t \|^2_{H^1(\T^d)},
\end{align}
In particular, from \eqref{eq:energy_equality} we deduce the following almost sure energy estimate, which we will use extensively in the following:
\begin{align} \label{eq:energy}
\sup_{t \in [0,1]} \left(\| T_t \|^2_{L^2(\T^d)}
+
2\kappa \int_0^t \| T_s \|^2_{H^1(\T^d)} ds \right)
\leq
\| T_0 \|^2_{L^2(\T^d)}.
\end{align} 
By \eqref{eq:energy_equality} and using the inequality $\| T_t \|^2_{L^2(\T^d)}\leq\| T_t \|^2_{H^1(\T^d)}$, one can deduce the following almost sure decay of $L^2(\T^d)$ norm for the solution of \eqref{eq:temp}:
\begin{align} \label{eq:energy_decay}
\| T_t \|_{L^2(\T^d)}
\leq 
e^{-\kappa t} \| T_0 \|_{L^2(\T^d)}.
\end{align}

In the inviscid case $\kappa=0$, we must introduce the flow $\mathbf{U}^\alpha$ associated to $\uu^\alpha$ to exhibit the (Lagrangian) solution $T_t=T_0 \circ (\mathbf{U}^\alpha)^{-1}$, which however need not to have trajectories in $L^2([0,T],H^1)$. Energy inequalities \eqref{eq:energy} and \eqref{eq:energy_decay} in this case read as $\| T_t \|_{L^2(\T^d)} \leq \| T_0 \|_{L^2(\T^d)}$.

Concerning equation \eqref{eq:temp_bar}, it is well known \cite[Theorem 5.2]{Pa83} that the operator $A$ generates an analytic semigroup of negative type $e^{A \cdot}$ on $\mathcal{S}'(\T^d)$, and the unique solution of \eqref{eq:temp_bar} is given by the Duhamel formula
\begin{align*}
\bar{T}_t = e^{At} T_0,
\quad
t \in (0,1).
\end{align*}
Decay of $L^2(\T^d)$ norm of $\bar{T}$ can be estimated as follows:
\begin{align*}
\| \bar{T}_t \|_{L^2(\T^d)}
\leq 
e^{-\lambda t} \|\bar{T}_0\|_{L^2(\T^d)},
\end{align*}
where $\lambda$ is the principal eigenvalue of the operator $-A$. Notice that $\lambda \geq \kappa$ since $\mathcal{L}$ is a negative-semidefinite operator.

\begin{remark} \label{rmk:transfer}
The inequality $\| T_t \|^2_{L^2(\T^d)}\leq\| T_t \|^2_{H^1(\T^d)}$, used to deduce \eqref{eq:energy_decay} above, may be very loose if the energy of $T$ is distributed at high wavenumbers, viz. the Fourier coefficients $\{\hat{T}_k\}_{k \in \Z^d_0}$ do not decrease sufficiently fast as $|k| \to \infty$.
This is in fact the case: indeed, by \autoref{thm:main}, for every fixed $t >0 $
\begin{align*}
\E{ \| \pi_N T_t \|_{L^2(\T^d)}}
&\leq
N \E{ \| T_t \|_{H^{-1}(\T^d)} }
\\
&\leq
N \left( \|\bar{T}_t \|_{H^{-1}(\T^d)}
+
\E{ \|T_t - \bar{T}_t \|_{H^{-1}(\T^d)} }
\right)
\\
&\lesssim
N \|T_0\|_{L^2(\T^d)}
\left( e^{-\lambda t} + \left( \epsilon +\frac{\mu^{2+\gamma}}{\alpha^\varkappa} \right)^\varsigma \right),
\end{align*}
where $\pi_N : \mathcal{S}'(\T^d) \to C^\infty(\T^d)$ denotes the projector onto Fourier modes $|k| \leq N$, $N \in \N$.
The previous inequality suggests that energy is actually transferred to high wavenumbers if $t>0$, $\epsilon \ll 1$ and $\lambda, \alpha \gg 1$, and hence we do not expect inequality \eqref{eq:energy_decay} to be sharp -- that is indeed the content of \autoref{thm:enhanced}. 
\end{remark}

\section{Useful estimates} \label{sec:est}
In this section, we prove some auxiliary results concerning the time increments of the process $T$ solution of \eqref{eq:temp}. Proofs are mostly inspired by our previous works \cite{AsFlPa20+,FlPa21}. 

Recall that $T$ has trajectories taking values in $L^\infty([0,1], L^2(\T^d))$ almost surely.
Forthcoming \autoref{lem:increments_T} and \eqref{eq:interp} allow to estimate the increments $T_{t+\delta} - T_t$ with respect to Sobolev norms $H^s(\T^d)$, $s \in [-1,0]$. 
It turns out that $T$ is actually a.s. H\"older continuous as a variable taking values in $H^s(\T^d)$ for $s \in (-1,0]$, and it is a.s. Lipschitz continuous as variable taking values in $H^{-1}(\T^d)$; however, its H\"older and Lipschitz constants diverge to infinity for $\alpha \to \infty$, and therefore we need \autoref{lem:increments_T_bis} to better control them in expected value.  

The subsequent \autoref{prop:discr_holder} aims to control the error between the actual solution $T$ of \eqref{eq:temp} and a discretized version of \eqref{eq:temp_bar}. 
We will make an essential use of the latter in the proof of \autoref{prop:holder_f} in \autoref{sec:main}.
Its proof, however, is quite long: for the sake of a clear and effective presentation we postpone it to \autoref{sec:proof}.

To start with, we recall the following useful lemma from \cite{JiZh20}. Let $p \geq 1$ and $j \in J$. Then the supremum of the Ornstein-Uhlenbeck process $\eta^{\alpha,j}$ can be estimated in expected value with
\begin{align} \label{eq:sup_OU}
\mathbb{E} \left[
\sup_{s \in [0,1]}
|\eta^{\alpha,j}(s)|^p
\right]
\lesssim_p \alpha^{p/2} \log(1+\alpha)^{p/2}.
\end{align}
To be precise, in \cite{JiZh20} the inequality is proved for a Ornstein-Uhlenbeck process with zero initial condition $\eta^{\alpha,j}(0)=0$. The general case is recovered using the decomposition
$\eta^{\alpha,j}(s) = \eta^{\alpha,j}(0) + (\eta^{\alpha,j}(s)-\eta^{\alpha,j}(0))$ and Gaussian estimates on $\eta^{\alpha,j}(0)$.
From \eqref{eq:sup_OU} and the definition of $\mu$,  one can deduce the following inequalities:
\begin{align}  \label{eq:sup_u1}
\E{
\sup_{s \in [0,1]}
\| \uu^\alpha (s) \|^p_{H^{d/2-\gamma}(\T^d)} 
}
\lesssim_p \mu^p \alpha^{p/2} \log^{p/2}(1+\alpha); 
\end{align}
\begin{align}  \label{eq:sup_u2}
\E{
\sup_{s \in [0,1]}
\| \uu^\alpha (s) \|^p_{L^\infty(\T^d)} 
}
\lesssim_p \mu^p \alpha^{p/2} \log^{p/2}(1+\alpha). 
\end{align}

We prove now the following result, which allows to control the time increments of the process $T$ in a Sobolev space of negative order. 
\begin{lemma} \label{lem:increments_T}
Let $\delta_{\star}>0$ such that
$\delta_{\star} \alpha \log(1+\alpha)> 1$.
Then for every $p\geq 1$ the following inequality holds:
\begin{align*}
\E{
\sup_{\substack{t+\delta \leq 1 \\ \delta \leq \delta_{\star}}}
\| T_{t+\delta} - T_t \|_{H^{-1}(\T^d)}^p
}
&\lesssim_p \| T_0 \|_{L^2(\T^d)}^p 
\mu ^p
\delta_{\star}^p
\alpha^{p/2} \log^{p/2}(1+\alpha).
\end{align*}
\end{lemma}
\begin{proof}
By the very definition of weak solution of \eqref{eq:temp} and \eqref{eq:energy}, for every test function $\phi \in \mathcal{D}(\T^d)$ one has the following almost sure inequality
\begin{align*}
|\langle \phi, T_{t+\delta} - T_t \rangle|
&\leq
\int_t^{t+\delta}
|\langle \uu^\alpha (s) \cdot \nabla \phi ,T_s \rangle| ds 
+ \kappa
\int_t^{t+\delta} 
|\langle \Delta \phi, T_s \rangle| ds
\\
&\leq 
\|\phi\|_{H^1(\T^d)}
\int_t^{t+\delta}
\| \uu^\alpha(s) \|_{L^\infty(\T^d)} 
\| T_s \|_{L^2(\T^d)} ds
\\
&\quad
+
\kappa
\|\phi\|_{H^1(\T^d)}  
\int_t^{t+\delta} 
\| T_s \|_{H^1(\T^d)} ds
\\
&\leq
\|\phi\|_{H^1(\T^d)}
\| T_0 \|_{L^2(\T^d)}
\delta
\sup_{s \in [0,1]} \left\| \uu^\alpha(s) \right\|_{L^\infty(\T^d)}
\\
&\quad
+
\|\phi\|_{H^1(\T^d)}
\| T_0 \|_{L^2(\T^d)}
\kappa^{1/2} \delta^{1/2}
.
\end{align*}
Since $\phi$ is arbitrary, we deduce
\begin{align*}
\|T_{t+\delta}-T_t\|_{H^{-1}(\T^d)}
&\leq 
\| T_0 \|_{L^2(\T^d)}
\left(
\delta
\sup_{s \in [0,1]} \left\| \uu^\alpha(s) \right\|_{L^\infty(\T^d)}
+
\kappa^{1/2} \delta^{1/2}
\right).
\end{align*}
Taking the supremum (raised to power $p$) over $\delta$, and then expectation, \eqref{eq:sup_u2} yields
\begin{align*}
\E{
\sup_{\substack{t+\delta \leq 1 \\ \delta \leq \delta_{\star}}}
\| T_{t+\delta} - T_t \|_{H^{-1}(\T^d)}^p
}
&\lesssim_p \| T_0 \|_{L^2(\T^d)}^p 
\left(
\delta_{\star}^p
\mu^p
\alpha^{p/2} \log^{p/2}(1+\alpha)
+
\kappa^{p/2} \delta^{p/2}_{\star}
\right).
\end{align*}
The thesis follows recalling that $\kappa^{p/2} \delta^{p/2}_{\star} < \delta_{\star}^p
\mu^p
\alpha^{p/2} \log^{p/2}(1+\alpha)$, due to our choice of parameters.
\end{proof}

The previous lemma can be used to deduce the following result. In view of interpolation inequality \eqref{eq:interp}, the next lemma can be used together with \autoref{lem:increments_T} in order to obtain better estimates -- needed in the following -- on the increments of $T$ in Sobolev spaces $H^s(\T^d)$, with $s \in (-2,-1)$.

\begin{lemma} \label{lem:increments_T_bis}
Let $d\geq 3$, $\delta\in (0,1)$ such that
$\delta\alpha \log(1+\alpha)> 1$ and $\mu \delta^{4/3} \alpha \log^{1/3}(1+\alpha)< 1$. Then for every $\gamma \in (0,(d-2)/2)$ and $p\geq 1$ the following inequality holds for every fixed $t \in [0,1-\delta]$:
\begin{align*}
\E{
\| T_{t+\delta} - T_t \|_{H^{-2-\gamma}(\T^d)}^p
}
&\lesssim_{\gamma,p}
\| T_0 \|^p_{L^2(\T^d)}
\mu^p
\left(
\delta^{p/2}
+ 
\alpha^{-p/2}  \log^{p/2}(1+\alpha)
\right)
.
\end{align*}
\end{lemma}
\begin{proof}
As in the proof of \autoref{lem:increments_T}, we have the following almost sure inequality for every given test function $\phi \in \mathcal{D}(\T^d)$
\begin{align*}
|\langle \phi, T_{t+\delta} - T_t \rangle|
&\leq
\int_t^{t+\delta}
|\langle \uu^\alpha(s) \cdot \nabla \phi, T_s-T_t \rangle| ds 
\\
&\quad+
\left|
\int_t^{t+\delta} 
\langle \uu^\alpha(s) \cdot \nabla \phi, T_t \rangle ds \right|
\\
&\quad 
+ \kappa
\int_t^{t+\delta} 
|\langle \Delta  \phi, T_s \rangle| ds.
\end{align*}
Let us deal with each term separately.
Using \autoref{lem:prod_bess} with $s_1 = d/2-\gamma$, $s_2 = 1+\gamma$ and \autoref{lem:increments_T} we get 
\begin{align*}
\int_t^{t+\delta}
|\langle \uu^\alpha(s) \cdot \nabla \phi, T_s-T_t \rangle| ds
&\leq
\int_t^{t+\delta}
\left\| \uu^\alpha(s)\cdot \nabla \phi \right\|_{H^{1}(\T^d)}
\left\| T_s-T_t \right\|_{H^{-1}(\T^d)} ds 
\\
&\lesssim_\gamma
\| \phi\|_{H^{2+\gamma}(\T^d)}
\int_t^{t+\delta}
\left\| \uu^\alpha(s)\right\|_{H^{d/2-\gamma}(\T^d)} \left\| T_s-T_t \right\|_{H^{-1}(\T^d)} ds,
\end{align*}
where we use $d/2-\gamma < d/2$, $1+\gamma < d/2$ to apply \autoref{lem:prod_bess}, which correspond to the conditions $d \geq 3$, $\gamma \in (0,(d-2)/2)$.

Moving to the next term, recall that $\uu^\alpha(s) ds = \sum_{j \in J} \uu_j \eta^{\alpha,j}(s) ds = \sum_{j\in J} \uu_j dW^j_s - \alpha^{-1} \sum_{j \in J} \uu_j d\eta^{\alpha,j}(s)$, and thus
\begin{align*}
\left|
\int_t^{t+\delta} 
\langle \uu^\alpha(s) \cdot \nabla \phi, T_t \rangle ds \right|
&=
\left|
\left\langle 
\left( \int_t^{t+\delta} 
\uu^\alpha(s) ds \right)
\cdot \nabla \phi, T_t \right\rangle \right|  
\\
&\leq
\| \phi \|_{H^1(\T^d)}
\| T_0 \|_{L^2(\T^d)} 
 \sum_{j \in J}
\| \uu_j \|_{L^\infty(\T^d)} 
\left| W^j_{t+\delta} - W^j_t \right|
\\
& \quad
+
\| \phi \|_{H^1(\T^d)}
\| T_0 \|_{L^2(\T^d)} 
\alpha^{-1}  
\sup_{s \in [0,1]} 
\| \uu^\alpha(s) \|_{L^\infty(\T^d)} .
\end{align*}
Finally, 
\begin{align*}
\int_t^{t+\delta} 
|\langle \Delta  \phi, T_s \rangle| ds
\leq 
\|\phi\|_{H^2(\T^d)}
\|T_0\|_{L^2(\T^d)} \delta.
\end{align*}
Therefore, since $\phi$ is arbitrary and $\|\phi\|_{H^1(\T^d)}, \|\phi\|_{H^2(\T^d)} \leq \|\phi\|_{H^{2+\gamma}(\T^d)}$ for every $\gamma>0$:
\begin{align*}
\| T_{t+\delta} - T_t \|_{H^{-2-\gamma}(\T^d)}
&\lesssim_\gamma
\sup_{s \in [0,1]} \|\uu^\alpha(s) \|_{H^{d/2-\gamma}(\T^d)}
\int_t^{t+\delta} \| T_s - T_t \|_{H^{-1}(\T^d)} ds
\\
&\quad +
\| T_0 \|_{L^2(\T^d)} 
 \sum_{j \in J}
\| \uu_j \|_{L^\infty(\T^d)} 
\left| W^j_{t+\delta} - W^j_t \right|
\\
&\quad+
\| T_0 \|_{L^2(\T^d)} 
\alpha^{-1}  
\sup_{s \in [0,1]} 
\| \uu^\alpha(s) \|_{L^\infty(\T^d)}
+
\| T_0 \|_{L^2(\T^d)} 
\delta.
\end{align*}
Using \eqref{eq:sup_u1}, \eqref{eq:sup_u2},  \autoref{lem:increments_T}, and the inequality $\E{\left| W^j_{t+\delta} - W^j_t \right|^p} \lesssim_p \delta^{p/2}$ valid for every $p\geq 1$ we get (recall that $\delta<1$ and therefore $\delta < \delta^{1/2}$)
\begin{align*}
\E{
\| T_{t+\delta} - T_t \|_{H^{-2-\gamma}(\T^d)}^p
}
&\lesssim_{\gamma,p}\| T_0 \|^p_{L^2(\T^d)} 
\mu^{2p} 
\delta^{2p}
\alpha^{p} \log^{p}(1+\alpha)
\\
&\quad+
\| T_0 \|^p_{L^2(\T^d)} 
\mu^p
\delta^{p/2}
\\
& \quad
+ \| T_0 \|^p_{L^2(\T^d)}
\mu^p \alpha^{-p/2}  \log^{p/2}(1+\alpha).
\end{align*}
The thesis now follows recalling our choice of parameters. 
\end{proof}

For the next proposition we need some preparation.
Divide the interval $[0,1]$ into subintervals of the form $[n\delta,(n+1)\delta]$, for some $\delta \in (0,1)$. In the following $\delta$ will be taken small, and it may depend on the parameter $\alpha$. The idea behind this subdivision is to control the following quantity:
\begin{align} \label{eq:discr_holder}
\sup_{\substack{n,m=1,\dots,1/\delta-1\\n>m}}
(n\delta-m\delta)^{-\theta}
\left| 
\langle \phi, T_{n\delta} \rangle
-
\langle \phi, T_{m\delta} \rangle
-
\delta \sum_{k=m}^{n-1}
\langle A \phi, T_{k\delta} \rangle 
\right|,
\end{align}
for any given test function $\phi \in \mathcal{D}(\T^d)$ and some $\theta>0$, where we have used the symbol $A$ as an abbreviation for the operator $\kappa\Delta + \mathcal{L}$. To be precise, we will prove the following:
\begin{proposition} \label{prop:discr_holder}
Let $d\geq 3$, $\beta>d/2+1$ and $\gamma \in (0,(d-2)/6)$ be fixed.
Then there exist $\theta>0$, $\varkappa >0$ and $\delta>0$ such that $1/\delta$ is an integer and for every $\alpha$ sufficiently large 
\begin{align*}
&\E{
\sup_{\substack{n,m=1,\dots,1/\delta-1\\n>m}}
(n\delta-m\delta)^{-\theta}
\left| 
\langle \phi, T_{n\delta} \rangle
-
\langle \phi, T_{m\delta} \rangle
-
\delta \sum_{k=m}^{n-1}
\langle A \phi, T_{k\delta} \rangle 
\right|}
\\
&\qquad\qquad\lesssim_\gamma
\| \phi \|_{H^{\beta}(\T^d)}
\| T_0 \|_{L^2(\T^d)}
\left( \epsilon + \frac{\mu^{2+\gamma}}{\alpha^{\varkappa}} 	\right)
.
\end{align*}
\end{proposition}

The proof of the previous result is quite long and technical, and we postpone it to \autoref{sec:proof}.
The particular choice of the parameters $\theta, \varkappa$ and $\delta$ is specified therein.
In the statement of the proposition, $\alpha$ sufficiently large means more precisely: $\mu \alpha^{-p_1} \log^{p_2}(1+\alpha)<1$, for some $p_1,p_2>0$ depending only on the parameters $\gamma, \theta$ and $\varkappa$.
More details are given in the proof of the proposition.
 
Moreover, since we always work at fixed $\gamma$, hereafter we omit the dependence of $\gamma$ in the symbol $\lesssim_\gamma$.

\section{Quantitative mixing and dissipation enhancement} \label{sec:main}

In the first part of this section we prove our main result \autoref{thm:main}. 
The idea is very simple, but effective.

Define the random distribution $f:\Omega \times [0,1] \to H^{-2}(\T^d)$ as follows: for every test function $\phi \in H^{2}(\T^d)$
\begin{align*}
\langle \phi, f_t \rangle
&=
\langle \phi, T_t \rangle
-
\langle \phi, T_0 \rangle
-
\int_0^t \langle A \phi, T_s \rangle ds,
\quad t \in [0,1].
\end{align*}
If one could prove $f=0$, then by \cite{Ba77} we would have $T_t - e^{A t} T_0 = T_t - \bar T_t = 0$ for every $t \in [0,1]$. Of course this is not the case, namely $f \neq 0$; however, owing to \cite[Theorem 1]{GuLeTi06} we can prove that the difference $T-\bar{T}$ is a continuous function of $f$ (with respect to suitable topologies), and therefore $T-\bar{T}$ is small if also $f$ is.

The ``right" topology in which to prove smallness of $f$ turns out to be that of H\"older continuous functions $C^\theta([0,1],H^{-\beta}(\T^d))$, for some small $\theta \in (0,1)$ and $\beta > d/2 + 1$ (cfr. \autoref{lem:mild_solution}). 
We would like to stress that \emph{any} $\theta \in (0,1)$ sufficiently small works: in particualr, it can be taken \emph{arbitrarily} small. This comes with no surprise: the reader familiar with SPDEs would recognize $f$ as a sort of additive noise perturbing the linear equation \eqref{eq:temp_bar}.

In the forthcoming \autoref{prop:holder_f} we provide suitable estimates for the H\"older norm of $f$.
As just discussed, we will use this result in the subsequent \autoref{ssec:proof_main} to prove \autoref{thm:main}.

Finally, we will prove \autoref{thm:enhanced} in \autoref{ssec:proof_enhanced}. We will deduce this theorem from \autoref{thm:main} and energy equality \eqref{eq:energy_equality}, using as an intermediate step Markov inequality to prove that the quantity $\sup_{t \in [0,1]}\| T_t - \bar T_t \|_{H^{-1}(\T^d)} $ is small with high probability.
 
\subsection{Estimate on $f$} \label{ssec:est_f}
 
We begin with the following remark: since $f_0=0$ and the time interval is compact, the H\"older norm of $f$ is equivalent to the H\"older seminorm
\begin{align*}
\| f \|_{C^\theta([0,1],H^{-\beta}(\T^d))}
\sim
\sup_{0<s<t<1}
\frac{\|f_t-f_s\|_{H^{-\beta}(\T^d)}}{|t-s|^{\theta}}.
\end{align*}

\begin{proposition} \label{prop:holder_f}
Let $d \geq 3$, $\beta>d/2+1$ and $\gamma \in (0,(d-2)/6)$.
Then there exists $\theta$ sufficiently small such that, for every $\alpha$ sufficiently large:
\begin{align*}
\E{\|f\|_{C^\theta([0,1],H^{-\beta}(\T^d))}}
\lesssim
\|T_0\|_{L^2(\T^d)}
\left( \epsilon + \frac{\mu^{2+\gamma}}{\alpha^{\varkappa}} 	\right)
\end{align*}
\end{proposition}
\begin{proof}
Let $s,t \in [0,1]$, $s<t$ and let $\phi \in \mathcal{D}(\T^d)$ be a test function.
Given $\delta$ as in \autoref{prop:discr_holder}, we distinguish two cases:
\begin{itemize}
\item
if $|t-s| \leq \delta$, then arguing as in the proof of \autoref{lem:increments_T} one can prove the following a.s. inequality:
\begin{align*}
\left| 
\langle \phi, (T_t-T_s) \rangle
\right|
&\leq
\|\phi\|_{H^{2}(\T^d)}
\|T_t-T_s\|_{H^{-2}(\T^d)}
\\
&\leq 
\|\phi\|_{H^{2}(\T^d)} 
\| T_0 \|_{L^2(\T^d)}
|t-s| 
\left( \sup_{s \in [0,1]} \| \uu^\alpha \|_{L^\infty(\T^d)} + \kappa \right)
\\
&\leq
\|\phi\|_{H^{2}(\T^d)} 
\| T_0 \|_{L^2(\T^d)}
|t-s|^\theta \delta^{1-\theta} \left( \sup_{s \in [0,1]} \| \uu^\alpha \|_{L^\infty(\T^d)} + \kappa \right).
\end{align*}
The previous inequality, together with
\begin{align*}
\left| 
\int_s^t \langle A \phi, T_r \rangle dr
\right|
&\lesssim
\|\phi\|_{H^{2}(\T^d)} 
\| T_0 \|_{L^2(\T^d)}
\mu^2 |t-s|,
\end{align*}
gives
\begin{align*}
\sup_{\substack{0<s<t<1,\\|t-s|\leq\delta}}
\frac{\|f_t-f_s\|_{H^{-\beta}(\T^d)}}{|t-s|^\theta}
\lesssim
\|T_0\|_{L^2(\T^d)}\delta^{1-\theta} 
\left( \sup_{s \in [0,1]} \| \uu^\alpha \|_{L^\infty(\T^d)} + \mu^2 \right).
\end{align*}
\item
if $|t-s|> \delta$, then there exist $n,m \in \N$, $n\geq m$, such that $t \in (n\delta,(n+1)\delta]$ and $s \in [(m-1)\delta,m\delta)$. Of course with this choice we have $|t-n\delta|,|n\delta-m\delta|,|m\delta-s| \leq |t-s|$, and therefore
\begin{align*}
\frac{\|f_t-f_s\|_{H^{-\beta}(\T^d)}}{|t-s|^\theta}
&\leq
\frac{\|f_t-f_{n\delta}\|_{H^{-\beta}(\T^d)}}{|t-n\delta|^\theta}
+
\frac{\|f_{n\delta}-f_{m\delta}\|_{H^{-\beta}(\T^d)}}{|n\delta-m\delta|^\theta}
\\
&\quad+
\frac{\|f_{m\delta}-f_s\|_{H^{-\beta}(\T^d)}}{|m\delta-s|^\theta}
,
\end{align*}
whenever $n>m$, and
\begin{align*}
\frac{\|f_t-f_s\|_{H^{-\beta}(\T^d)}}{|t-s|^\theta}
&\leq
\frac{\|f_t-f_{n\delta}\|_{H^{-\beta}(\T^d)}}{|t-n\delta|^\theta}
+
\frac{\|f_{n\delta}-f_s\|_{H^{-\beta}(\T^d)}}{|n\delta-s|^\theta}
,
\end{align*}
in the special case $n=m$. Since $|t-n\delta|,|m\delta-s| \leq \delta$, we can estimate
\begin{align*}
\sup_{0<s<t<1}
\frac{\|f_t-f_s\|_{H^{-\beta}(\T^d)}}{|t-s|^\theta}
&\lesssim
\sup_{\substack{0<s<t<1,\\|t-s|\leq\delta}}
\frac{\|f_t-f_s\|_{H^{-\beta}(\T^d))}}{|t-s|^\theta}
\\ 
&\quad+
\sup_{\substack{n,m=1,\dots,1/\delta-1\\n>m}}
\frac{\|f_{n\delta}-f_{m\delta}\|_{H^{-\beta}(\T^d)}}{|n\delta-m\delta|^\theta}.
\end{align*} 
\end{itemize} 
Putting all together, we finally get
\begin{align*}
\|f\|_{C^\theta([0,1],H^{-\beta}(\T^d))}
&\lesssim
\|T_0\|_{L^2(\T^d)} \delta^{1-\theta}\left( \sup_{s \in [0,1]} \| \uu^\alpha \|_{L^\infty(\T^d)} + \mu^2 \right)
\\ 
&\quad+
\sup_{\substack{n,m=1,\dots,1/\delta-1\\n>m}}
\frac{\|f_{n\delta}-f_{m\delta}\|_{H^{-\beta}(\T^d)}}{|n\delta-m\delta|^\theta},
\end{align*}
and therefore the thesis follows by \eqref{eq:sup_u2}, \autoref{prop:discr_holder} and the choice of $\delta$, whenever $\theta$ is sufficiently small.
\end{proof}

\subsection{Proof of \autoref{thm:main}} \label{ssec:proof_main}

We are ready to prove our main result \autoref{thm:main}.
In the first place, we recall \cite[Theorem 1]{GuLeTi06}, that in our setting reads as follows:
\begin{lemma} \label{lem:mild_solution}
Let $\theta \in (0,1)$.
Then for all $\vartheta< \theta$ there exists a linear map
\begin{align*}
\mathfrak{S}: 
C^\theta([0,1],H^{-\beta}(\T^d)) \to
C^\vartheta([0,1],H^{-\beta}(\T^d))
\end{align*}
that associates to every $X \in C^\theta([0,1],H^{-\beta}(\T^d))$ the unique weak solution $u$ of the evolution equation
\begin{align*}
u_t - \int_0^t A u_s ds
=
X_t,
\qquad u_0 = 0.
\end{align*} 

Moreover, 
\begin{align*}
\| \mathfrak{S}(X) \|_{C^\vartheta([0,1],H^{-\beta}(\T^d))}
&\lesssim
\| X \|_{C^\theta([0,1],H^{-\beta}(\T^d))}.
\end{align*}
\end{lemma}

\begin{proof}[Proof of \autoref{thm:main}]
By definition of $f$, the process $T_t$ is almost surely the weak solution of 
\begin{align*}
T_t - T_0 - \int_0^t A T_s ds = f_t,
\end{align*}
whereas $\bar{T}$ solves
\begin{align*}
\bar{T}_t - T_0 - \int_0^t A \bar{T}_s ds = 0.
\end{align*}
Applying \autoref{lem:mild_solution} with $u = T - \bar{T}$ and $X = f$ we get the almost sure inequality
\begin{align} \label{eq:T-barT}
\| T - \bar{T} \|_{C^\vartheta([0,1],H^{-\beta}(\T^d))}
\lesssim
\| f \|_{C^\theta([0,1],H^{-\beta}(\T^d))}.
\end{align}

Recall that, by \autoref{prop:holder_f}, there exists $\theta$ sufficiently small such that, for every $\alpha$ sufficiently large 
\begin{align*}
\E{\|f\|_{C^{\theta}([0,1],H^{-\beta}(\T^d))}}
&\lesssim
\| T_0 \|_{L^2(\T^d)}  
\left(
\epsilon
+\frac{\mu^{2+\gamma}}{\alpha^\varkappa}
\right);
\end{align*}
hence, the thesis of the theorem follows for every $s \geq \beta$, while for $s \in (0,\beta)$, \eqref{eq:energy}, \eqref{eq:T-barT} and interpolation inequality \eqref{eq:interp} yield
\begin{align*}
\E{ 
\| T - \bar{T} \|_{C^\vartheta([0,1],H^{-s}(\T^d))} }
&\leq
\E{ 
\| T_0 \|_{L^2(\T^d)}^{1-s/\beta}
\| T - \bar{T} \|_{C^\vartheta([0,1],H^{-\beta}(\T^d))}^{s/\beta} }
\\
&\lesssim
\E {
\| T_0 \|_{L^2(\T^d)}^{1-s/\beta}
\| f \|_{C^\theta([0,1],H^{-\beta}(\T^d))}^{s/\beta}
}
\\
&\leq
\| T_0 \|_{L^2(\T^d)}^{1-s/\beta} 
\E {
\| f \|_{C^\theta([0,1],H^{-\beta}(\T^d))}
} ^{s/\beta}
\\
&\lesssim
\| T_0 \|_{L^2(\T^d)}  
\left(
\epsilon
+\frac{\mu^{2+\gamma}}{\alpha^\varkappa}
\right)^{s/\beta}.
\end{align*} 
The proof is complete.
\end{proof}

\subsection{Proof of \autoref{thm:enhanced}}
\label{ssec:proof_enhanced}

In this paragraph we are concerned with the proof of 
\autoref{thm:enhanced}, that quantifies dissipation enhancement of $L^2(\T^d)$ for the solution of \eqref{eq:temp} when $\kappa>0$. 
We insist once again that our result states that a suitable velocity field $\uu^\alpha$ can dissipate energy almost \emph{instantaneously}, i.e. for every fixed $t>0$, without the necessity of taking $t$ large enough to have $\E{\|T_t\|_{L^2(\T^d)}}$ below a certain threshold. 

\begin{proof}[Proof of \autoref{thm:enhanced}]
The following proof is mostly inspired by the proof of \cite[Theorem 1.4]{BeBlPu20}.
By energy equality \eqref{eq:energy_equality} and interpolation inequality \eqref{eq:interp}, we have
\begin{align} \label{eq:d/dt}
\frac{d}{dt} \|T_t\|_{L^2(\T^d)}^2
=
-2\kappa \|T_t\|^2_{H^1(\T^d)}
\leq
-2\kappa \frac{\|T_t\|_{L^2(\T^d)}^4}{\|T_t\|^2_{H^{-1}(\T^d)}}.
\end{align}
From the previous inequality it is clear that, in order to control $\| T_t \|_{L^2(\T^d)}$, it is sufficient to have a good bound from above on the quantity $\| T_t \|_{H^{-1}(\T^d)}$.
Notice that the trivial bound $\| T_t \|_{H^{-1}(\T^d)} \leq \| T_t \|_{L^2(\T^d)}$ produces the equally trivial estimate \eqref{eq:energy_decay}. To have a better control on $\| T_t \|_{H^{-1}(\T^d)}$ we will exploit \autoref{thm:main} as follows.

Denote $c = C^{1/2}\left(\epsilon+\frac{\mu^{2+\gamma}}{\alpha^\varkappa}\right)^{\varsigma/2}>0$, as in the statement of the theorem.
By \autoref{thm:main} and Markov inequality we have
\begin{align} \label{eq:markov}
\PP \left\{
\sup_{t \in [0,1]}
\| T_t - \bar T_t \|_{H^{-1}(\T^d)} > 
c \|T_0\|_{L^2(\T^d)}
\right\}
&\leq
c,
\end{align}
hence with probability at least $1-c$ it holds
\begin{align*}
\|T_t\|^2_{H^{-1}(\T^d)} 
\leq
2\|\bar T_t\|^2_{H^{-1}(\T^d)}
+ 
2\|T_t - \bar T_t\|^2_{H^{-1}(\T^d)}
\leq
2\|T_0\|_{L^2(\T^d)}^2
\left( e^{-2\lambda t} + c^2 \right).
\end{align*}
In view of this, \eqref{eq:d/dt} implies on a set of probability at least $1-c$
\begin{align*}
\| T_t \|_{L^2(\T^d)}^2
\leq
\frac{\| T_0 \|_{L^2(\T^d)}^2}
{1+\kappa \int_0^t \frac{ds}{e^{-2\lambda s} + c^2}}
=
\frac{\| T_0 \|_{L^2(\T^d)}^2}
{1+\frac{\kappa}{2 \lambda c^2}\log \left( \frac{c^2 e^{2 \lambda t} + 1}{c^2+1} \right) },
\end{align*}
where in the last equality we have used $c>0$.
This completes the proof of the first part of the theorem. As for the second part, it immediately follows from the a.s. inequality \eqref{eq:energy_decay}, \eqref{eq:markov} and
\begin{align*}
\E{ \| T_t \|_{L^2(\T^d)} }
&=
\E{ \| T_t \|_{L^2(\T^d)} 
\mathbf{1}_{\left\{
\sup_{t \in [0,1]}
\| T_t - \bar T_t \|_{H^{-1}(\T^d)} > 
c \|T_0\|_{L^2(\T^d)}
\right\}} }
\\
&\quad
+
\E{ \| T_t \|_{L^2(\T^d)} 
\mathbf{1}_{\left\{
\sup_{t \in [0,1]}
\| T_t - \bar T_t \|_{H^{-1}(\T^d)} \leq c \|T_0\|_{L^2(\T^d)}
\right\}} }.
\end{align*}
\end{proof}

\begin{remark}
Looking back at \eqref{eq:d/dt} one realizes that an alternative approach could be that of producing a lower bound for $\| T_t \|_{H^{1}(\T^d)}$ instead of an upper bound for $\| T_t \|_{H^{-1}(\T^d)}$.
In order to do this, we present an heuristic argument. Write
\begin{align*}
\| T_t \|^2_{H^{1}(\T^d)}
&=
\| \pi_N T_t \|^2_{H^{1}(\T^d)} +
\| (I-\pi_N) T_t \|^2_{H^{1}(\T^d)}
\\
&\geq
\| \pi_N T_t \|^2_{L^2(\T^d)} +
N^2 \| (I-\pi_N) T_t \|^2_{L^2(\T^d)}
\\
&=
(1-N^2)\| \pi_N T_t \|^2_{L^2(\T^d)} +
N^2 \| T_t \|^2_{L^2(\T^d)},
\end{align*}
where $\pi_N$ denotes the Fourier projector onto modes $|k| \leq N$, $N \in \N$.
Plugging into \eqref{eq:d/dt}, and assuming $N^2 \gg \kappa^{-1}$, we have formally
\begin{align} \label{eq:aux_integral}
\|T_t\|_{L^2(\T^d)}^2
&\leq
e^{-2\kappa N^2 t} \|T_0\|_{L^2(\T^d)}^2
+
\int_0^t
e^{-2\kappa N^2 (t-s)}
2\kappa (N^2-1) \| \pi_N T_s \|^2_{L^2(\T^d)} ds
\\
&\sim \nonumber
e^{-2\kappa N^2 t} \|T_0\|_{L^2(\T^d)}^2
+
\| \pi_N T_t \|^2_{L^2(\T^d)}
.
\end{align}
Using the inequality above with $N^2=\lambda/\kappa$, $\lambda\gg 1$ and recalling \autoref{rmk:transfer} 
\begin{align*}
\E{ \| \pi_N T_t \|_{L^2(\T^d)}}
&\lesssim
N \|T_0\|_{L^2(\T^d)}
\left( e^{-\lambda t} + \left( \epsilon +\frac{\mu^{2+\gamma}}{\alpha^\varkappa} \right)^\varsigma \right),
\end{align*}
gives
\begin{align*}
\E{\|T_t\|_{L^2(\T^d)}}
\lesssim
\frac{\lambda^{1/2}}{\kappa^{1/2}}\left( e^{-\lambda t} + \left( \epsilon +\frac{\mu^{2+\gamma}}{\alpha^\varkappa} \right)^\varsigma \right)
\|T_0\|_{L^2(\T^d)}.
\end{align*}
Comparing with \eqref{eq:enhanced_remark}, the previous estimate has in addition the term $\left( \epsilon +\frac{\mu^{2+\gamma}}{\alpha^\varkappa} \right)^\varsigma$ on the right hand side, and the implicit constant in the inequality.
Moreover, it has been recovered assuming $\lambda \gg 1$ in order to approximate the time integral in \eqref{eq:aux_integral} with $\| \pi_N T_t \|^2_{L^2(\T^d)}$.
On the other hand, the statement of \autoref{thm:enhanced}, and in particular \eqref{eq:dissipation}, is valid for every value of $\lambda$.
\end{remark}

\section{Proof of \autoref{prop:discr_holder}}
\label{sec:proof}

In this section we give the proof of \autoref{prop:discr_holder}. Recall that we are interested in the expression \eqref{eq:discr_holder}, given by 
\begin{align*}
\sup_{\substack{n,m=1,\dots,1/\delta-1\\n>m}}
(n\delta-m\delta)^{-\theta}
\left| 
\langle \phi, T_{n\delta} \rangle
-
\langle \phi, T_{m\delta} \rangle
-
\delta \sum_{k=m}^{n-1}
\langle A \phi, T_{k\delta} \rangle 
\right|,
\end{align*}
where $\phi \in \mathcal{D}(\T^d)$ is a smooth test function, $\theta>0$ and $A=\kappa\Delta + \mathcal{L}$.

The choice of the parameter $\delta$ in the expression above is crucial. We will see that, in order to have the result of \autoref{prop:discr_holder}, the parameters $\delta$ and $\alpha$ must satisfy very particular relations.

\subsection{A convenient decomposition}
In order to control \eqref{eq:discr_holder}, we first consider the quantity
\begin{align*}
\langle \phi, T_{n\delta} \rangle
-
\langle \phi, T_{m\delta} \rangle
-
\delta \sum_{k=m}^{n-1}
\langle A \phi, T_{k\delta} \rangle,
\qquad n>m,
\end{align*}
or equivalently
\begin{align*}
\langle \phi, T_{(n+1)\delta} \rangle
-
\langle \phi, T_{m\delta} \rangle
-
\delta \sum_{k=m}^{n}
\langle A \phi, T_{k\delta} \rangle,
\qquad n \geq m.
\end{align*}
Let us preliminarily rewrite the previous expression in a more convenient way.
For every $k=0,\dots,1/\delta-1$ it holds
\begin{align} \label{eq:incr01}
\langle \phi, T_{(k+1)\delta}  \rangle
-
\langle \phi, T_{k\delta} \rangle 
&=  
\int_{k\delta}^{(k+1)\delta}
\langle \uu^\alpha(s) \cdot \nabla \phi,  T_s \rangle ds
+ \kappa
\int_{k\delta}^{(k+1)\delta}
\langle \Delta\phi, T_s \rangle ds  
\\
&= \nonumber
I_1(k) + I_2(k).
\end{align}

Let us further decompose (cfr. for instance \cite[Equation 18]{AsFlPa20+})
\begin{align*}
I_1(k) &= 
\int_{k\delta}^{(k+1)\delta}
\langle \uu^\alpha(s) \cdot \nabla \phi,  T_s \rangle ds
\\
&=
\int_{k\delta}^{(k+1)\delta}
\langle \uu^\alpha(s) \cdot \nabla \phi,  
(T_s-T_{k\delta}) \rangle ds
+
\int_{k\delta}^{(k+1)\delta}
\langle \uu^\alpha(s) \cdot \nabla \phi,  
T_{k\delta} \rangle ds
\\
&=
\int_{k\delta}^{(k+1)\delta}
\int_{k\delta}^s
\langle \uu^\alpha(r) \cdot \nabla
(\uu^\alpha(s) \cdot \nabla \phi),  
(T_r-T_{k\delta}) \rangle dr ds
\\
&\quad
+
\int_{k\delta}^{(k+1)\delta}
\int_{k\delta}^s
\langle \uu^\alpha(r) \cdot \nabla
(\uu^\alpha(s) \cdot \nabla \phi),  
T_{k\delta} \rangle dr ds
\\
&\quad
+ \kappa
\int_{k\delta}^{(k+1)\delta}
\int_{k\delta}^s
\langle \Delta(\uu^\alpha(s) \cdot \nabla \phi),  
T_r \rangle dr ds
\\
&\quad
+
\int_{k\delta}^{(k+1)\delta}
\langle \uu^\alpha(s) \cdot \nabla \phi,  
T_{k\delta} \rangle ds
\\
&=
I_{11}(k) + I_{12}(k) + I_{13}(k) + I_{14}(k),
\end{align*}
where we have used $\uu^\alpha(s) \cdot \nabla \phi \in \mathcal{D}(\T^d)$ as a test function in order to pass from the second to the third line. 
Indeed, since $\dvg \uu^\alpha(s) = 0$ for every $s \in [0,1]$ we have
\begin{align*}
\langle 1 , \uu^\alpha(s) \cdot \nabla \phi \rangle
=
\langle 1 , \dvg (\uu^\alpha(s) \phi) \rangle
= -
\langle \nabla 1 , \uu^\alpha(s) \phi \rangle
= 0.
\end{align*}
The term $I_{12}(k)$ can be rewritten as follows:
\begin{align*}
I_{12}(k)
&=
\int_{k\delta}^{(k+1)\delta}
\int_{k\delta}^s
\langle \uu^\alpha(r) \cdot \nabla
(\uu^\alpha(s) \cdot \nabla \phi),  
T_{k\delta} \rangle dr ds
\\
&=
\sum_{j,j' \in J}
\int_{k\delta}^{(k+1)\delta}
\int_{k\delta}^s
\langle \uu_{j'} \eta^{j',\alpha}(r) \cdot \nabla
(\uu_{j} \eta^{j,\alpha}(s) \cdot \nabla \phi),  
T_{k\delta} \rangle dr ds
\\
&=
\sum_{j,j' \in J}
\langle \uu_{j'}  \cdot \nabla
(\uu_{j} \cdot \nabla \phi),  
T_{k\delta} \rangle 
\int_{k\delta}^{(k+1)\delta}
\int_{k\delta}^s
\eta^{\alpha,j}(s)
\eta^{\alpha,j'}(r)
dr ds
\\
&=
\sum_{j,j' \in J}
\langle \uu_{j'}  \cdot \nabla
(\uu_{j} \cdot \nabla \phi),  
T_{k\delta} \rangle 
\left(
\int_{k\delta}^{(k+1)\delta}
\int_{k\delta}^s
\eta^{\alpha,j}(s)
\eta^{\alpha,j'}(r)
dr ds - \delta_{j,j'}\frac{\delta}{2} \right)
\\
&\quad+
\delta
\langle \mathcal{L} \phi, T_{k\delta} \rangle 
\\
&=
I_{121}(k) + I_{122}(k)
.
\end{align*}
As for the term $I_2(k)$ in \eqref{eq:incr01}, we have
\begin{align*}
I_{2}(k)
&=
\kappa
\int_{k\delta}^{(k+1)\delta}
\langle \Delta \phi, T_s \rangle ds
\\
&=
\kappa
\int_{k\delta}^{(k+1)\delta}
\langle \Delta \phi, (T_s - T_{k\delta} )\rangle ds
+\kappa
\int_{k\delta}^{(k+1)\delta}
\langle \Delta \phi, T_{k\delta} \rangle ds
\\
&=
I_{21}(k)+I_{22}(k).
\end{align*}

Taking the sum of \eqref{eq:incr01} over $k=m,\dots,n$ we get
\begin{align} \label{eq:incr02}
\langle \phi, T_{(n+1)\delta} \rangle
&-
\langle \phi, T_m \rangle
-
\delta \sum_{k=m}^{n}
\langle A \phi, T_{k\delta} \rangle
\\
&=
\sum_{k=m}^{n}
\left( 
I_{11}(k) + I_{121}(k) + I_{13}(k) + I_{14}(k) + I_{21}(k)
\right). \nonumber
\end{align}

\subsection{Controlling the terms $I_{11}(k)$, $I_{13}(k)$ and $I_{21}(k)$}

\begin{lemma} \label{lem:negligible}
Let $d \geq 3$ and $\gamma \in (0,(d-2)/6)$.
Let $\delta>0$ be such that
$\delta\alpha \log(1+\alpha)> 1$ and $\mu \delta^{4/3} \alpha \log^{1/3}(1+\alpha)< 1$.
Then the following estimates hold:
\begin{align*}
\mathbb{E} \left[
\sup_{\substack{n,m=1,\dots,1/\delta-1\\n\geq m}} \left|
\sum_{k=m}^{n} I_{11}(k)
\right| \right] 
&\lesssim
\left\| \phi \right\|_{H^{2+3\gamma}(\T^d)}
\left\| T_0 \right\|_{L^2(\T^d)} 
\mu^{2+\gamma}
\delta^{1+\gamma} 
\alpha^{1+\gamma/2} \log^{1+\gamma/2}(1+\alpha)
;
\\
\mathbb{E} \left[
\sup_{\substack{n,m=1,\dots,1/\delta-1\\n\geq m}} \left|
\sum_{k=m}^{n} I_{13}(k)
\right| \right]
&\lesssim
\|  \phi \|_{H^{2+2\gamma}(\T^d)}
\|T_0\|_{L^2(\T^d)}
\mu 
\delta^{(1+\gamma)/2}
\alpha^{1/2} \log^{1/2}(1+\alpha) 
; 
\\
\mathbb{E} \left[
\sup_{\substack{n,m=1,\dots,1/\delta-1\\n\geq m}} \left|
\sum_{k=m}^{n} I_{21}(k)
\right| \right]
&\lesssim
\| \phi \|_{H^{2+\gamma}(\T^d)}
\|T_0\|_{L^2(\T^d)}
\mu^\gamma 
\delta^{\gamma}
\alpha^{\gamma/2} \log^{\gamma/2}(1+\alpha). 
\end{align*}
\end{lemma}

\begin{proof}

Throughout the proof, we will use without explicit mention the following key inequality:
\begin{align*} 
\mathbb{E} \left[
\sup_{\substack{n,m=1,\dots,1/\delta-1\\n\geq m}} \left|
\sum_{k=m}^{n} I(k)
\right| \right] 
\leq
\sum_{k=1}^{1/\delta-1} 
\mathbb{E} \left[
 \left|
I(k)
\right| \right] ,
\quad
I = I_{11}, I_{13}, I_{21}.
\end{align*}

Let us start from $I_{11}(k)$:
\begin{align*}
I_{11}(k) =
\int_{k\delta}^{(k+1)\delta}
\int_{k\delta}^s
\langle \uu^\alpha(r) \cdot \nabla
(\uu^\alpha(s) \cdot \nabla \phi),  
(T_r-T_{k\delta}) \rangle dr ds.
\end{align*}
Using \autoref{lem:prod_bess} and \eqref{eq:interp}, for every $\gamma \in (0,(d-2)/6)$
\begin{align*}
|I_{11}(k)| 
&\leq
\int_{k\delta}^{(k+1)\delta}
\int_{k\delta}^s
\left\|  \uu^\alpha(s) \cdot \nabla \phi 
\right\|_{H^{1+2\gamma}(\T^d)}  
\left\| \uu^\alpha(r) (T_r- T_{k\delta}) \right\|_{H^{-2\gamma}(\T^d)} dr ds
\\
&\lesssim
\int_{k\delta}^{(k+1)\delta}
\int_{k\delta}^s
\left\|  \uu^\alpha(s) \right\|_{H^{d/2-\gamma}(\T^d)}
\left\| \phi \right\|_{H^{2+3\gamma}(\T^d)} 
\left\|  \uu^\alpha(r) \right\|_{H^{d/2-\gamma}(\T^d)} 
\left\| (T_r- T_{k\delta}) \right\|_{H^{-\gamma}(\T^d)} dr ds
\\
&\leq 
\left\| \phi \right\|_{H^{2+3\gamma}(\T^d)} 
\sup_{s \in [0,1]} \left\|  \uu^\alpha(s) \right\|_{H^{d/2-\gamma}(\T^d)}^2
\left\| T_0 \right\|_{L^2(\T^d)}^{1-\gamma}
\int_{k\delta}^{(k+1)\delta}
\int_{k\delta}^s
\left\| (T_r- T_{k\delta}) \right\|_{H^{-1}(\T^d)}^{\gamma} dr ds
. 
\end{align*}
Hence \autoref{lem:increments_T} and \eqref{eq:sup_u1} yield
\begin{align*}
\mathbb{E} \left[
\sup_{\substack{n,m=1,\dots,1/\delta-1\\n\geq m}} \left|
\sum_{k=m}^{n} I_{11}(k)
\right| \right] 
&\leq
\sum_{k=1}^{1/\delta-1}
\mathbb{E} \left[
\left| I_{11}(k) \right| \right] 
\\
&\lesssim 
\left\| \phi \right\|_{H^{2+3\gamma}(\T^d)}
\left\| T_0 \right\|_{L^2(\T^d)} 
\mu^{2+\gamma}
\delta^{1+\gamma} 
\alpha^{1+\gamma/2} \log^{1+\gamma/2}(1+\alpha)
.
\end{align*}

As for the term $I_{13}(k)$ we have
\begin{align*}
I_{13}(k)
&=
\kappa
\int_{k\delta}^{(k+1)\delta}
\int_{k\delta}^s
\langle \Delta(\uu^\alpha(s) \cdot \nabla \phi), T_r \rangle dr ds.
\end{align*}
By \autoref{lem:prod_bess}, H\"older inequality and \eqref{eq:energy}, for every $\gamma \in (0,(d-2)/4)$
\begin{align*}
|I_{13}(k)|
&\leq \kappa
\int_{k\delta}^{(k+1)\delta}
\int_{k\delta}^s
\| \uu^\alpha(s) \cdot \nabla \phi \|_{H^{1+\gamma}(\T^d)} 
\| T_r \|_{H^{1-\gamma}(\T^d)} dr ds
\\
&\lesssim \kappa
\int_{k\delta}^{(k+1)\delta}
\int_{k\delta}^s
\| \uu^\alpha(s) \|_{H^{d/2-\gamma}(\T^d)}
\|  \phi \|_{H^{2+2\gamma}(\T^d)} 
\|T_r \|_{H^{1-\gamma}(\T^d)} dr ds
\\
&\leq \kappa 
\|  \phi \|_{H^{2+2\gamma}(\T^d)} 
\sup_{s \in [0,1]} \| \uu^\alpha(s) \|_{H^{d/2-\gamma}(\T^d)}
\|T_0\|_{L^2(\T^d)}^{\gamma}
\int_{k\delta}^{(k+1)\delta}
\int_{k\delta}^s
\|T_r \|_{H^{1}(\T^d)}^{1-\gamma} dr ds
\\
&\leq \kappa^{(1+\gamma)/2} 
\|  \phi \|_{H^{2+2\gamma}(\T^d)} 
\sup_{s \in [0,1]} \| \uu^\alpha(s) \|_{H^{d/2-\gamma}(\T^d)}
\|T_0\|_{L^2(\T^d)}
\delta^{(3+\gamma)/2}.
\end{align*}
Hence using \eqref{eq:sup_u1}
\begin{align*} 
\mathbb{E} \left[
\sup_{\substack{n,m=1,\dots,1/\delta-1\\n\geq m}} \left|
\sum_{k=m}^{n} I_{13}(k)
\right| \right] 
&\leq
\sum_{k=1}^{1/\delta-1}
\mathbb{E} \left[
\left| I_{13}(k) \right| \right] 
\\
&\lesssim
\|  \phi \|_{H^{2+2\gamma}(\T^d)}
\|T_0\|_{L^2(\T^d)}
\mu 
\delta^{(1+\gamma)/2}
\alpha^{1/2} \log^{1/2}(1+\alpha) 
.
\end{align*}

Let us move finally to the term $I_{21}(k)$:
\begin{align*}
I_{21}(k) &= 
\kappa
\int_{k\delta}^{(k+1)\delta}
\langle \Delta \phi , (T_s-T_{k\delta})
\rangle ds.
\end{align*}
By \eqref{eq:interp} we have
\begin{align*}
|I_{21}(k)|
&\leq \kappa
\int_{k\delta}^{(k+1)\delta}
\| \phi \|_{H^{2+\gamma}(\T^d)}
\| T_s-T_{k\delta} \|_{H^{-\gamma}(\T^d)} ds
\\
&\leq \kappa \|T_0\|_{L^2(\T^d)}^{1-\gamma}
\int_{k\delta}^{(k+1)\delta}
\| \phi \|_{H^{2+\gamma}(\T^d)}
\| T_s-T_{k\delta} \|^\gamma_{H^{-1}(\T^d)} ds.
\end{align*}
Hence using \eqref{eq:sup_u2} and \autoref{lem:increments_T} we obtain
\begin{align*}
\mathbb{E} \left[
\sup_{\substack{n,m=1,\dots,1/\delta-1\\n\geq m}} \left|
\sum_{k=m}^{n} I_{21}(k)
\right| \right] 
&\leq
\sum_{k=1}^{1/\delta-1}
\mathbb{E} \left[
\left| I_{21}(k) \right| \right] 
\\
&\lesssim
\| \phi \|_{H^{2+\gamma}(\T^d)}
\|T_0\|_{L^2(\T^d)}
\mu^\gamma 
\delta^{\gamma}
\alpha^{\gamma/2} \log^{\gamma/2}(1+\alpha).
\end{align*}
\end{proof}

\subsection{Controlling the term $I_{121}(k)$}

The term $I_{121}(k)$ requires more care.
We will deduce estimates for this term using a martingale argument due to Nakao, that can be found for instance in \cite{IkWa81}.
Recall
\begin{align*}
I_{121}(k)
=
\sum_{j,j' \in J}
\langle \uu_{j'}  \cdot \nabla
(\uu_{j} \cdot \nabla \phi),  
T_{k\delta} \rangle 
\left(
\int_{k\delta}^{(k+1)\delta}
\int_{k\delta}^s
\eta^{\alpha,j}(s)
\eta^{\alpha,j'}(r)
dr ds - \delta_{j,j'}\frac{\delta}{2} \right).
\end{align*}
Define the following quantity
\begin{align*}
c_{j,j'}(k)
&=
\int_{k\delta}^{(k+1)\delta}
\int_{k\delta}^s
\eta^{\alpha,j}(s)
\eta^{\alpha,j'}(r)
dr ds.
\end{align*}
By the explicit expression of the Ornstein-Uhlenbeck process \eqref{eq:explicit_OU}, the conditional expectation of $c_{j,j'}(k)$ with respect to $\mathcal{F}_{k\delta}$ gives 
\begin{align*}
\mathbb{E} \left[
c_{j,j'}(k)
\mid \mathcal{F}_{k\delta} \right]
&=
\int_{k\delta}^{(k+1)\delta}
\int_{k\delta}^s
\mathbb{E} \left[
\eta^{\alpha,j}(s)
\eta^{\alpha,j'}(r)
\mid \mathcal{F}_{k\delta} \right]dr ds
\\
&=
\eta^{\alpha,j}(k\delta)
\eta^{\alpha,j'}(k\delta)
\int_{k\delta}^{(k+1)\delta}
\int_{k\delta}^s
e^{-\alpha(s-k\delta)}
e^{-\alpha(r-k\delta)}
dr ds
\\
&\quad
+
\delta_{j,j'}
\int_{k\delta}^{(k+1)\delta}
\int_{k\delta}^s
\mathbb{E} \left[
\alpha^2
\int_{k\delta}^s
e^{-\alpha(s-s')}  dW^{j}_{s'}
\int_{k\delta}^r
e^{-\alpha(r-r')} dW^{j}_{r'}
\right] dr ds
\\
&=
\eta^{\alpha,j}(k\delta)
\eta^{\alpha,j'}(k\delta)
\int_{k\delta}^{(k+1)\delta}
\int_{k\delta}^s
e^{-\alpha(s-k\delta)}
e^{-\alpha(r-k\delta)}
dr ds
\\
&\quad
+
\delta_{j,j'}
\alpha^2
\int_{k\delta}^{(k+1)\delta}
\int_{k\delta}^s
\int_{k\delta}^r
e^{-\alpha(s-r')} 
e^{-\alpha(r-r')}  d{r'}
dr ds
\\
&=
\eta^{\alpha,j}(k\delta)
\eta^{\alpha,j'}(k\delta)
\frac{\alpha^{-2}}{2}
\left(
1-e^{-\alpha\delta}
\right)^2
\\
&\quad
+
\delta_{j,j'}
\left(
\frac{\delta}{2} +
\alpha^{-1} \left(
e^{-\alpha\delta} - 1
+ \frac{1}{4} 
\left(1- e^{-2\alpha \delta} \right) \right) \right).
\end{align*}
We introduce now the following auxiliary processes:
\begin{align*}
M_n
&=
\sum_{k=1}^{n-1}
\sum_{j,j' \in J}
\langle \uu_{j'}  \cdot \nabla
(\uu_{j} \cdot \nabla \phi),  
T_{k\delta} \rangle 
\left( c_{j,j'}(k)
- \mathbb{E} \left[ 
c_{j,j'}(k) \mid \mathcal{F}_{k\delta}\right] \right).
\\
R_n
&=
\sum_{k=1}^{n-1}
\sum_{j,j' \in J}
\langle \uu_{j'}  \cdot \nabla
(\uu_{j} \cdot \nabla \phi),  
T_{k\delta} \rangle 
\left( \mathbb{E} \left[ 
c_{j,j'}(k) \mid \mathcal{F}_{k\delta}\right] - 		\delta_{j,j'}\frac{\delta}{2} \right).
\end{align*}
Since $T_{k\delta}$ is $\mathcal{F}_{k\delta}$-measurable, the process $\{M_n\}_{n=1,\dots,1/\delta}$ is a discrete martingale with respect to the filtration $\mathcal{G}_n \coloneqq \mathcal{F}_{(n-1)\delta}$ with initial condition $M_1=0$. By Doob maximal inequality and the martingale property we have the following
\begin{align*}
&\mathbb{E} \left[ 
\sup_{n = 1, \dots, 1/\delta}
M_n^2 \right]
\lesssim
\mathbb{E} \left[ M_{1/\delta}^2 \right]
=
\sum_{k=1}^{1/\delta-1}
\mathbb{E} \left[
|M_{k+1}- M_{k}|^2 \right]
\\
&\quad=
\sum_{k=1}^{1/\delta-1}
\E{
\left|
\sum_{j,j' \in J}
\langle \uu_{j'}  \cdot \nabla
(\uu_{j} \cdot \nabla \phi),  
T_{k\delta} \rangle
\left( c_{j,j'}(k)
- \mathbb{E} \left[ 
c_{j,j'}(k) \mid \mathcal{F}_{k\delta}\right] \right)
\right|^2
}
.
\end{align*}
Using the inequality, valid for $\gamma \in (0,(d-2)/2)$,
\begin{align*}
\left|
\langle \uu_{j'}  \cdot \nabla
(\uu_{j} \cdot \nabla \phi),  
T_{k\delta} \rangle
\right|
&\leq
\| \uu_j \cdot \nabla \phi \|_{H^{1}(\T^d)}
\| \uu_{j'} T_{k\delta} \|_{L^2(\T^d)}
\\
&\lesssim
\| \uu_j \|_{H^{d/2-\gamma}(\T^d)}
\| \phi \|_{H^{2+\gamma}(\T^d)}
\| \uu_{j'} \|_{L^\infty(\T^d)}
\| T_0 \|_{L^2(\T^d)},
\end{align*}
we arrive at the following estimate
\begin{align*}
&\left|
\sum_{j,j' \in J}
\langle \uu_{j'}  \cdot \nabla
(\uu_{j} \cdot \nabla \phi),  
T_{k\delta} \rangle 
\left( c_{j,j'}(k)
- \E {
c_{j,j'}(k) \mid \mathcal{F}_{k\delta}
} \right)
\right|^2
\\
&\lesssim
\sum_{j,j' \in J}
\| \phi \|_{H^{2+\gamma}(\T^d)}^2
\| T_0 \|_{L^2(\T^d)}^2
\| \uu_j \|_{H^{d/2-\gamma}(\T^d)}^2
\| \uu_{j'} \|_{L^\infty(\T^d)}^2
\\
&\qquad \times
\sum_{j,j' \in J}
\left( c_{j,j'}(k)
- \mathbb{E} \left[ 
c_{j,j'}(k) \mid \mathcal{F}_{k\delta}\right] \right)^2
\\
&=
\| \phi \|_{H^{2+\gamma}(\T^d)}
\| T_0 \|_{L^2(\T^d)}
\sum_{j \in J}
\| \uu_j \|_{H^{d/2-\gamma}(\T^d)}^2
\sum_{j' \in J}
\| \uu_{j'} \|_{L^\infty(\T^d)}^2
\\
&\qquad \times 
\sum_{j,j' \in J}
\left( c_{j,j'}(k)
- \mathbb{E} \left[ 
c_{j,j'}(k) \mid \mathcal{F}_{k\delta}\right] \right)^2.
\end{align*}
Since the conditional expectation is a $L^2(\Omega)$-projection,
\begin{align*}
\mathbb{E} \left[ 
\sup_{n = 1, \dots, 1/\delta}
|M_n| \right]
&\leq
\mathbb{E} \left[ 
\sup_{n = 1, \dots, 1/\delta}
M_n^2 \right]^{1/2}
\\
&\quad\lesssim
\| \phi \|_{H^{2+\gamma}(\T^d)}
\| T_0 \|_{L^2(\T^d)} 
\left(
\sum_{j \in J}
\| \uu_j \|_{H^{d/2-\gamma}(\T^d)}^2
\sum_{j' \in J}
\| \uu_{j'} \|_{L^\infty(\T^d)}^2
\right)^{1/2}
\\
&\qquad \times
\left(
\sum_{k=1}^{1/\delta-1}
\sum_{j,j' \in J}
\E{
\left( c_{j,j'}(k)
- \mathbb{E} \left[ 
c_{j,j'}(k) \mid \mathcal{F}_{k\delta}\right] \right)^2
}
\right)^{1/2}
\\
&\quad\leq
\| \phi \|_{H^{2+\gamma}(\T^d)}
\| T_0 \|_{L^2(\T^d)} 
\left(
\sum_{j \in J}
\| \uu_j \|_{H^{d/2-\gamma}(\T^d)}^2
\sum_{j' \in J}
\| \uu_{j'} \|_{L^\infty(\T^d)}^2
\right)^{1/2}
\\
&\qquad \times
\left(
\sum_{k=1}^{1/\delta-1}
\sum_{j,j' \in J}
\E{
c_{j,j'}(k)^2
}
\right)^{1/2}
\\
&\quad\lesssim
\| \phi \|_{H^{2+\gamma}(\T^d)}
\| T_0 \|_{L^2(\T^d)}
\mu^2
\delta \alpha^{1/2} \log^{1/2}(1+\alpha),
\end{align*}
where the last inequality follows from
\begin{align*}
\E{
c_{j,j'}(k)^2
}
&=
\E{
\left(
\int_{k\delta}^{(k+1)\delta}
\int_{k\delta}^s
\eta^{\alpha,j}(s)
\eta^{\alpha,j'}(r)
dr ds
\right)^2
}
\\
&=
\E{
\left(
\int_{k\delta}^{(k+1)\delta}
\eta^{\alpha,j}(s)
\left( W^{j'}_s - W^{j'}_{k\delta} - \alpha^{-1}
\left( \eta^{\alpha,j'}(s) - \eta^{\alpha,j'}(k\delta) \right) \right)
ds
\right)^2
}
\\
&\lesssim
\delta
\int_{k\delta}^{(k+1)\delta}
\E{
\sup_{s \in [0,1]} |\eta^{\alpha,j'}(s)|^2
\left( |W^{j'}_s - W^{j'}_{k\delta}|^2 - \alpha^{-2}
\sup_{s \in [0,1]} |\eta^{\alpha,j'}(s)|^2 \right)
}
ds
\\
&\lesssim 
\delta^3 \alpha \log(1+\alpha)
+ 
\delta^2 \log(1+\alpha).
\end{align*}
As for the remaining term,
\begin{align*}
&\mathbb{E} \left[ 
\sup_{n = 1, \dots, 1/\delta}
|R_n| \right]
\\
&\quad\lesssim
\| \phi \|_{H^{2+\gamma}(\T^d)}
\| T_0 \|_{L^2(\T^d)} 
\sum_{j \in J}
\| \uu_j \|_{H^{d/2-\gamma}(\T^d)}
\sum_{j' \in J}
\| \uu_{j'} \|_{L^\infty(\T^d)}
\\
&\qquad \times
\sum_{k=1}^{1/\delta-1}
\sum_{j,j' \in J}
\E{
\left|  \mathbb{E} \left[ 
c_{j,j'}(k) \mid \mathcal{F}_{k\delta}\right] - \delta_{j,j'}\frac{\delta}{2}\right|
}
\\
&\quad\lesssim
\| \phi \|_{H^{2+\gamma}(\T^d)}
\| T_0 \|_{L^2(\T^d)}
\mu^2
\delta^{-1} \alpha^{-1} \log(1+\alpha),
\end{align*}
where we have used
\begin{align*}
\E{
\left|  \mathbb{E} \left[ 
c_{j,j'}(k) \mid \mathcal{F}_{k\delta}\right] - \delta_{j,j'}\frac{\delta}{2}\right|
}
\lesssim
\alpha^{-1} \log(1+\alpha).
\end{align*}

Putting all together, and recalling 
$
\sum_{k=m}^{n} I_{121}(k)
=
M_{n+1} + R_{n+1}
-M_m-R_m
$, 
we deduce the following: 
\begin{lemma} \label{lem:negligible_bis}
Let $d$, $\gamma$, $\delta$ and $\alpha$ as in \autoref{lem:negligible}.
Then
\begin{align*}
\E{
\sup_{\substack{n,m=1,\dots,1/\delta-1\\n\geq m}} \left|
\sum_{k=m}^{n} I_{121}(k)
\right| 
}
&\lesssim
\| \phi \|_{H^{2+\gamma}(\T^d)}
\| T_0 \|_{L^2(\T^d)}
\mu^2
\delta^{-1} \alpha^{-1} 
\log(1+\alpha) .
\end{align*}
\end{lemma}

\subsection{Controlling the remaining terms}
Consider now the term $I_{14}(k)$ in \eqref{eq:incr02}:
\begin{align*}
I_{14}(k) 
&=
\int_{k\delta}^{(k+1)\delta}
\langle \uu^\alpha(s) \cdot \nabla \phi, 
T_{k\delta}  \rangle ds  
=
\left\langle
\left( \int_{k\delta}^{(k+1)\delta}
\uu^\alpha(s) ds \right) \cdot \nabla \phi, 
T_{k\delta}  \right\rangle .
\end{align*}
Recall $\uu^\alpha(s) ds = \sum_{j \in J} \uu_j \eta^{\alpha,j}(s) ds = \sum_{j\in J} \uu_j dW^j_s - \alpha^{-1} \sum_{j \in J} \uu_j d\eta^{\alpha,j}(s)$, and thus
\begin{align*}
I_{14}(k) 
&=
\sum_{j \in J} 
\langle \uu_j \cdot \nabla \phi, 
T_{k\delta}  \rangle 
\left( 
W^j_{(k+1)\delta}-
W^j_{k\delta}
\right)
\\
&\quad
-
\sum_{j \in J} 
\langle \uu_j \cdot \nabla \phi, 
T_{k\delta}  \rangle
\alpha^{-1} 
\left( 
\eta^{\alpha,j}((k+1)\delta)-
\eta^{\alpha,j}(k\delta) \right)
\\
&=
\sum_{j \in J} 
\langle \uu_j \cdot \nabla \phi, 
T_{k\delta}  \rangle 
\left( 
W^j_{(k+1)\delta}-
W^j_{k\delta}
\right)
-
\sum_{j \in J} 
\int_{k\delta}^{(k+1)\delta}
\langle \uu_j \cdot\nabla \phi, 
T_{s} \rangle  dW^j_s
\\
&\quad
+
\sum_{j \in J} \int_{k\delta}^{(k+1)\delta}
\langle \uu_j \cdot \nabla \phi, 
T_{s} \rangle  dW^j_s
\\
&\quad
-
\sum_{j \in J} 
\langle \uu_j \cdot \nabla \phi,  
T_{k\delta}  \rangle
\alpha^{-1} 
\left( 
\eta^{\alpha,j}((k+1)\delta)-
\eta^{\alpha,j}(k\delta) \right)
\\
&=
\sum_{j \in J} 
\int_{k\delta}^{(k+1)\delta}
\langle \uu_j \cdot \nabla \phi, 
T_{k\delta}- T_{s} \rangle  dW^j_s
\\
&\quad
+
\sum_{j \in J} \int_{k\delta}^{(k+1)\delta}
\langle \uu_j \cdot \nabla \phi, 
T_{s} \rangle  dW^j_s
\\
&\quad
-
\sum_{j \in J} 
\langle \uu_j \cdot \nabla \phi, 
T_{k\delta}  \rangle
\alpha^{-1} 
\left( 
\eta^{\alpha,j}((k+1)\delta)-
\eta^{\alpha,j}(k\delta) \right)
\\
&=
I_{141}(k)+I_{142}(k)+I_{143}(k).
\end{align*}

We have the forthcoming:
\begin{lemma} \label{lem:negligible_tris}
Let $d$, $\gamma$, $\delta$ and $\alpha$ as in \autoref{lem:negligible} and denote $\theta = \frac{1+\gamma}{1+2\gamma}$.
Then
\begin{align*}
\E{
\sup_{\substack{n,m=1,\dots,1/\delta-1\\n\geq m}} \left|
\sum_{k=m}^{n} I_{141}(k)
\right| 
}
&\lesssim
\|\phi\|_{H^{2+\gamma}(\T^d)}
\|T_0\|_{L^2(\T^d)}
\mu^2 \delta \alpha^{1/2}\log^{1/2}(1+\alpha);
\\
\E{
\sup_{\substack{n,m=1,\dots,1/\delta-1\\n\geq m}} \left|
\sum_{k=m}^{n} I_{143}(k)
\right| 
}
&\lesssim
\|\phi\|_{H^{2+2\gamma}(\T^d)}
\|T_0\|_{L^2(\T^d)} 
\mu^{2}
\\
&\quad \times
\left(
\delta^{(\theta-1)/2}
\alpha^{(\theta-1)/2}
\log^{(1+\theta)/2}(1+\alpha)
+
\delta^{\theta-1}
\alpha^{\theta-1}
\log(1+\alpha)
\right)
.
\end{align*}
\end{lemma}
\begin{proof}
Concerning the term $I_{141}(k)$, we have
\begin{align*}
\sum_{k=m}^{n} I_{141}(k)
&=
\sum_{j \in J} 
\int_{m\delta}^{(n+1)\delta}
\langle \uu_j \cdot \nabla \phi, 
T_{[s]}- T_{s} \rangle  dW^j_s,
\end{align*}
where we denote by $[s]$ the largest multiple of $\delta$ smaller than $s$.
Therefore by Burkholder-Davis-Gundy inequality and \autoref{lem:increments_T},
\begin{align*}
\E{ \sup_{\substack{n,m=1,\dots,1/\delta-1\\n\geq m}}
\left| \sum_{k=m}^{n} I_{141}(k)\right|}
&\lesssim
\E{ \sup_{t \in [0,1]}
\left|
\sum_{j \in J} 
\int_{0}^{t}
\langle \uu_j \cdot \nabla \phi, 
T_{[s]}- T_{s} \rangle  dW^j_s \right|}
\\
&\lesssim
\E{ 
\left|
\sum_{j \in J} 
\int_0^1
\langle \uu_j  \cdot \nabla \phi, 
T_{[s]}- T_{s} \rangle^2  ds \right|^{1/2}}
\\
&\lesssim
\E{ 
\left|  
\int_0^1
\sum_{j \in J}\| \uu_j \|_{H^{d/2-\gamma}(\T^d)}^2
\|\phi\|_{H^{2+\gamma}(\T^d)}^2 
\|T_{[s]}- T_{s}\|_{H^{-1}(\T^d)}^2  ds \right|^{1/2}}
\\
&\lesssim
\|\phi\|_{H^{2+\gamma}(\T^d)}
\|T_0\|_{L^2(\T^d)}
\mu^2 \delta \alpha^{1/2}\log^{1/2}(1+\alpha) .
\end{align*}

Let us move now to $I_{143}(k)$.
Since the increments of $\eta^{\alpha,j}$ are difficult to control, we perform a discrete integration by parts to get
\begin{align*}
\sum_{k=m}^{n} I_{143}(k)
&=
-
\sum_{k=m}^{n}
\sum_{j \in J} 
\langle \uu_j \cdot \nabla \phi, 
T_{k\delta}  \rangle
\alpha^{-1} 
\left( 
\eta^{\alpha,j}((k+1)\delta)-
\eta^{\alpha,j}(k\delta) \right)
\\
&=
-
\alpha^{-1}
\sum_{k=m}^{n}
\langle \left( 
\uu^\alpha((k+1)\delta)-
\uu^\alpha(k\delta) \right) \cdot \nabla \phi, 
T_{k\delta}  \rangle
\\
&=
\alpha^{-1}
\sum_{k=m+1}^{n} 
\langle \uu^\alpha(k\delta) \cdot \nabla \phi, 
(T_{k\delta}-T_{(k-1)\delta})  \rangle
\\
&\quad
-
\alpha^{-1}
\sum_{j \in J} 
\langle \uu^\alpha(m\delta) \cdot \nabla \phi,  
T_{m\delta}  \rangle 
\\
&\quad
-
\alpha^{-1}
\langle \uu^\alpha ((n+1)\delta) \cdot \nabla \phi, 
T_{n\delta}  \rangle. 
\end{align*}
By the usual estimates, taking $\gamma\in(0,(d-2)/4) $ we have
\begin{align*}
\left|\sum_{k=m}^{n} I_{143}(k) \right|
&\lesssim
\alpha^{-1}
\|\phi\|_{H^{2+2\gamma}(\T^d)}
\sup_{s \in [0,1]}
\|\uu^\alpha(s)\|_{H^{d/2-\gamma}(\T^d)}
\\
&\quad \times 
\left(
\|T_0\|_{L^2(\T^d)}
+
\sum_{k=m+1}^{n}
\|T_{k\delta}-T_{(k-1)\delta}\|_{H^{-1-\gamma}(\T^d)}
\right)
\\
&\eqqcolon \sum_{k=m}^{n} I'_{143}(k).
\end{align*}
Interpolation inequality \eqref{eq:interp} with $\theta = \frac{1+\gamma}{1+2\gamma}$ gives:
\begin{align*}
\|T_{k\delta}-T_{(k-1)\delta}\|_{H^{-1-\gamma}(\T^d)}
\leq
\|T_{k\delta}-T_{(k-1)\delta}\|_{H^{-1}(\T^d)}
^\theta
\|T_{k\delta}-T_{(k-1)\delta}\|_{H^{-2-2\gamma}(\T^d)}
^{1-\theta},
\end{align*}
and therefore by \autoref{lem:increments_T} and \autoref{lem:increments_T_bis}:
\begin{align*}
\E{
\sup_{\substack{n,m=1,\dots,1/\delta-1\\n\geq m}}
\left|\sum_{k=m}^{n} I_{143}(k) \right|
}
&\leq
\E{
\sum_{k=1}^{1/\delta-1} I'_{143}(k) 
}
\\
&\lesssim
\|\phi\|_{H^{2+2\gamma}(\T^d)}
\|T_0\|_{L^2(\T^d)} 
\mu^{2}
\\
&\quad \times
\left(
\delta^{(\theta-1)/2}
\alpha^{(\theta-1)/2}
\log^{(1+\theta)/2}(1+\alpha)
+
\delta^{\theta-1}
\alpha^{\theta-1}
\log(1+\alpha)
\right)
.
\end{align*}

\end{proof}

\subsection{Proof of \autoref{prop:discr_holder}}

In this paragraph we are going to prove \autoref{prop:discr_holder}. We want to show
\begin{align*}
&\E{
\sup_{\substack{n,m=1,\dots,1/\delta-1\\n>m}}
(|n-m|\delta)^{-\theta}
\left| 
\langle \phi, T_{k\delta} \rangle
-
\langle \phi, T_{m\delta} \rangle
-
\delta \sum_{k=m}^{n-1}
\langle A \phi, T_{k\delta} \rangle 
\right|}
\\
&\quad\lesssim
\| \phi \|_{H^{\beta}(\T^d)}
\| T_0 \|_{L^2(\T^d)}
\left( \epsilon + \frac{\mu^{2+\gamma}}{\alpha^{\varkappa}} 	\right)
,
\end{align*}
for some $\theta, \varkappa$ sufficiently small and $\alpha, \beta, \gamma, \delta$ as in the statement of the proposition.
Let us preliminarily discuss the condition on $\delta$. 
First, in order to apply \autoref{lem:increments_T_bis}, the parameter $\delta$ must be chosen depending on $\alpha$ (and $\mu$) so that
\begin{align} \label{eq:condition_da1}
\delta\alpha \log(1+\alpha)> 1,
\quad
\mu \delta^{4/3} \alpha \log^{1/3}(1+\alpha)< 1.
\end{align}
Moreover, recall the following decomposition \eqref{eq:incr02} 
\begin{align*}
\langle \phi, T_{n\delta} \rangle
&-
\langle \phi, T_{m\delta} \rangle
-
\delta \sum_{k=m}^{n-1}
\langle A \phi, T_{k\delta} \rangle
\\
&=
\sum_{k=m}^{n-1}
\left( 
I_{11}(k) + I_{121}(k) + I_{13}(k) + I_{141}(k) + I_{142}(k) + I_{143}(k) + I_{21}(k)
\right). \nonumber
\end{align*}
We are assuming $1/\delta$ to be an integer, so that the previous decomposition is well-calibrated -- the interval $[0,1]$ is split exactly into $1/\delta$ subintervals of length $\delta$.

To simplify the notation, write $I(k)$ as an abbreviation for $
|I_{11}(k)| + |I_{12}(k)| + |I_{13}(k)| + |I_{141}(k)| + |I_{143}(k)| + |I_{21}(k)|$. The only term remaining is that involving $|I_{142}(k)|$, that will be treated separately.
We shall prove next that for every $\theta$ sufficiently small there exists $\varkappa>0$ such that
\begin{align} \label{eq:aux_varkappa}
&\E{
\sup_{\substack{n,m=1,\dots,1/\delta-1\\n>m}}
(|n-m|\delta)^{-\theta}
\sum_{k=m}^{n-1}
I(k)
}
\lesssim
\| \phi \|_{H^{2+\gamma}(\T^d)}
\| T_0 \|_{L^2(\T^d)}
\frac{\mu^{2+\gamma}}{\alpha^{\varkappa}}.
\end{align}
Invoking \autoref{lem:negligible}, and in particular the estimate for the term involving $I_{11}$, one immediately realizes that for the previous estimate to be true it is necessary that
\begin{align} \label{eq:condition_da2}
\delta^{1+\gamma-\theta} \alpha^{1+\gamma/2+\varkappa} \log^{1+\gamma/2}(1+\alpha) < 1.
\end{align}

Then, once \eqref{eq:condition_da1} and \eqref{eq:condition_da2} are both satisfied, from \autoref{lem:negligible}, \autoref{lem:negligible_bis} and \autoref{lem:negligible_tris} we deduce \eqref{eq:aux_varkappa} (possibly taking smaller $\theta$ and $\varkappa$ if needed).
Let $\theta$, $\varkappa$ such that
\begin{align*}
1 < \frac{1+\gamma-\theta}{1+\gamma/2+\varkappa},
\end{align*}
which is always possible taking $\theta$ and $\varkappa$ sufficiently small. Then $\delta$ is chosen by
\begin{align*}
\delta = c_1 \alpha^{-c_2},
\quad
\max \left\{\frac{4}{5} , \frac{1+\gamma/2+\varkappa}{1+\gamma-\theta}\right\} < c_2 < 1,
\end{align*}
and $c_1$ is an auxiliary constant such that $1/\delta$ is an integer. 
With such a choice of $\delta$, conditions \eqref{eq:condition_da1} and \eqref{eq:condition_da2} hold true for large $\alpha$, since logarithmic factors become negligible when compared with powers of $\alpha$.

Let us move next to the term $\sum_{k=m}^{n-1}I_{142}(k)$, given by
\begin{align*}
\sum_{k=m}^{n-1}I_{142}(k)
=
\sum_{j \in J} \int_{m\delta}^{n\delta} 
\langle \uu_j \cdot \nabla \phi, T_r \rangle
dW^j_r.
\end{align*}

Since by Sobolev embedding Theorem $\| \nabla \phi \|_{L^\infty(\T^d)} \lesssim \|\phi\|_{H^\beta(\T^d)}$ for every $\beta > d/2+1$, Burkholder-Davis-Gundy inequality gives
\begin{align*}
\E{
\left|
\sum_{j \in J} \int_s^t 
\langle \uu_j \cdot \nabla \phi, T_r \rangle
dW^j_r \right|^3
}
&\lesssim
\E{
\left|
\sum_{j \in J} \int_s^t 
|\langle \uu_j \cdot \nabla \phi, T_r \rangle|^2
dr \right|^{3/2}
}
\\
&\lesssim 
|t-s|^{3/2} 
\epsilon^3
\| \phi \|_{H^{\beta}(\T^d)}^3 
\| T_0 \|_{L^2(D)}^3.
\end{align*}
Therefore, by Kolmogorov continuity criterion the stochastic integral in the expression above is $\theta$-H\"older continuous for every $\theta < 1/6$, and its H\"older constant $K_\theta$ satisfies
\begin{align*}
\E{
\sup_{0<s<t<1}
\frac{\left| \sum_{j \in J} \int_s^t 
\langle \uu_j \cdot \nabla \phi, T_r \rangle
dW^j_r \right|}{|t-s|^\theta}
}
=
\E{K_\theta}
\lesssim
\epsilon
\| \phi \|_{H^{\beta}(\T^d)}
 \| T_0 \|_{L^2(\T^d)}.
\end{align*}
The proof is complete.

\bibliographystyle{plain}

\end{document}